\documentclass[a4paper,10.9pt]{amsart}

\DeclareRobustCommand{\SkipTocEntry}[5]{}

\usepackage{ amsmath }
\usepackage{marvosym}

\usepackage[T1]{fontenc}
\usepackage{empheq}
\usepackage{relsize}
\usepackage{amsmath}

\usepackage{bm}

\usepackage{upgreek}
\usepackage{ esint }
\usepackage{color}
\usepackage{comment}
\usepackage{amssymb}
\usepackage{amsfonts}
\usepackage{graphicx}

\usepackage{amscd}

\usepackage{slashed}
\usepackage{pdflscape}
\usepackage{tikz}
\usepackage{enumerate}
\usepackage{multirow}
\usepackage{stmaryrd}
\usepackage{xcolor}

\usepackage{scalerel,stackengine}

\usepackage{ifthen}

\usepackage{bbold}
\usepackage[linkcolor=blue]{hyperref}
\usepackage{mathrsfs}
\usepackage[colorinlistoftodos]{todonotes}

\definecolor{blue}{rgb}{.255,.41,.884} 
\definecolor{red}{rgb}{1, 0, 0} 
\definecolor{green}{rgb}{.196,.804,.196} 
\definecolor{yellow}{rgb}{1,.648,0} 
\definecolor{pink}{rgb}{1,0.5,0.5}

\newtheorem{theorem}{Theorem}[section]
\newtheorem{lemma}[theorem]{Lemma}

\newtheorem{proposition}[theorem]{Proposition}
\newtheorem{corollary}[theorem]{Corollary}

\theoremstyle{definition}
\newtheorem{definition}[theorem]{Definition}
\newtheorem{example}[theorem]{Example}

\theoremstyle{remark}
\newtheorem{remark}[theorem]{Remark}

\newtheorem{problem}[theorem]{Problem}

\usepackage{pdfsync}

\usepackage{graphicx} 
\usepackage{amsmath} 
\usepackage{amsfonts}
\usepackage{amssymb}

\newcommand{\be}{\begin{equation}}

\newcommand{\ee}{\end{equation}}

\newcommand{\boldnabla}{\mbox{\boldmath$ \nabla$}}

\newcommand{\al}{\mbox{\boldmath$\Delta$}}

\usepackage{ulem}

\newlength{\myuldp}
\newlength{\myulthickness}
\newcommand{\setmyul}[2]{%
  \if\relax\detokenize{#1}\relax\else
    \setlength{\myuldp}{#1}%
  \fi
  \setlength{\myulthickness}{#2}%
}
\newcommand{\setmyuldepth}[1]{%
  \settodepth{\myuldp}{%
    \if\relax\detokenize{#1}\relax\myulalphabet\else#1\fi
  }%
  \addtolength{\myuldp}{0.0pt}
}
\newcommand{\myulalphabet}{abcdefghijklmnopqrstuvwxyz}

\setmyul{0pt}{0.4pt}
\setmyuldepth{}

\setmyul{-0.25pt}{.3pt}

\makeatletter
\newcommand*{\ovF}[1]{%
  $\m@th\overline{\raisebox{0pt}[\dimexpr\height-1.5pt\relax]{#1}}$%
}

\newcommand*{\mathringdn}[1]{%
  $\m@th\mathring{\raisebox{0pt}[\dimexpr\height-.6pt\relax]{#1}}$%
}

\makeatother

\newcommand{\II}{{\rm  I\hspace{-.2mm}I}}
\newcommand{\IIo}{\hspace{0.4mm}\mathring{\rm{ I\hspace{-.2mm} I}}{\hspace{.0mm}}}

\newcommand{\IIIo}{{\mathring{{\bf\rm I\hspace{-.2mm} I \hspace{-.2mm} I}}{\hspace{.2mm}}}{}}
\newcommand{\IVo}{{\mathring{{\bf\rm I\hspace{-.2mm} V}}{\hspace{.2mm}}}{}}
\newcommand{\Vo}{{\mathring{{\bf\rm V}}}{}}

\newcommand{\FF}[1]{\hspace{.2mm}\raisebox{0.07pt}{$
	{\mathring{\myul{$\overbracket[0.37pt][-1pt]{\hspace{-0.3mm} \scalebox{1}{
\resizebox{!}{\heightof{${\rm I}$}}{\raisebox{-0.1pt}{{{\resizebox{\widthof{$#1$}}{\height}{\phantom{d}}}\llap{$\mathbf{#1}$}}}}
}\hspace{-0.3mm}}$}}}$}\hspace{.2mm}}

\renewcommand{\FF}[1]{\hspace{.2mm}\raisebox{0.07pt}{$
	{\mathring{\myul{\ovF{\hspace{-0.6mm} \scalebox{1}{
\resizebox{!}{\heightof{${\rm I}$}}{\raisebox{-0.1pt}{{{\resizebox{\widthof{$#1$}}{\height}{\phantom{d}}}\llap{$\mathbf{#1}$}}}}
}\hspace{-0.3mm}}}}}$}\hspace{.2mm}}

\renewcommand{\FF}[1]{\mathring{\underline{\overline{\rm{#1}}}}}

\newcommand{\otop}{\mathring{\top}}

\newcommand{\R}{{\bf R }}

\newcommand{\ba}{\begin{array}}

\newcommand{\ea}{\end{array}}

\newcommand{\beq}{\begin{eqnarray}}

\newcommand{\eeq}{\end{eqnarray}}

\newtheorem{lm}{lemma}

\newtheorem{thee}{theorem}

\newtheorem{proo}{proposition}

\newtheorem{co}{corollary}

\newtheorem{rem}{remark}

\newtheorem{deff}{definition}

\newcommand{\bd}{\begin{deff}}

\newcommand{\ed}{\end{deff}}

\newcommand{\bl}{\begin{lm}}

\newcommand{\el}{\end{lm}}

\newcommand{\bp}{\begin{proo}}

\newcommand{\ep}{\end{proo}}

\newcommand{\bt}{\begin{thee}}

\newcommand{\et}{\end{thee}}

\newcommand{\bc}{\begin{co}}

\newcommand{\ec}{\end{co}}

\newcommand{\brm}{\begin{rem}}

\newcommand{\erm}{\end{rem}}

\usepackage{t1enc}
\def\frak{\mathfrak}

\def\Cal{\mathcal}

\newcommand{\bS}{\mathbb{S}}

\newcommand{\newc}{\newcommand}

\renewcommand{\Im}{\operatorname{Im}}

\let\ccdot.

\newc{\aR}{\mbox{\boldmath{$ R$}}}
\newc{\aS}{\mbox{\boldmath{$ S$}}}
\newc{\aT}{\mbox{\boldmath{$ T$}}}
\newc{\aW}{\mbox{\boldmath{$ W$}}}

\newc{\aD}{\mbox{\boldmath{$ D$}}\hspace{-.2mm}}

\newcommand{\aRic}{\mbox{\boldmath{$ \Ric$}}}

\newc{\aK}{\mbox{\boldmath{$ K$}}}
\newc{\aL}{\mbox{\boldmath{$ L$}}}

\newcommand{\ce}{{\Cal E}}

\newcommand{\ct}{{\Cal T}}

\newcommand{\bT}{{\Bbb T}}

\usepackage{amssymb}
\usepackage{amscd}

\newcommand{\Rho}{{\it P}}

\newcommand{\Ric}{{\it Ric}}

\let\hash=\sharp  

\newcommand{\aM}{\, \widetilde{\!M}\, }

\newcommand{\nn}[1]{(\ref{#1})}

\newcommand{\X}{\mbox{\boldmath{$ X$}}}
 
\newcommand{\h}{\mbox{\boldmath{$ h$}}}

\newc{\obstrn}[2]{B^{#1}_{#2}}

\newcommand{\rpl}                         
{\mbox{$
\begin{picture}(12.7,8)(-.5,-1)
\put(0,0.2){$+$}
\put(4.2,2.8){\oval(8,8)[r]}
\end{picture}$}}

\newcommand{\lpl}                         
{\mbox{$
\begin{picture}(12.7,8)(-.5,-1)
\put(2,0.2){$+$}
\put(6.2,2.8){\oval(8,8)[l]}
\end{picture}$}}

\newc{\tensor}[1]{#1}
\newc{\Mvariable}[1]{\mbox{#1}}
\newc{\down}[1]{{}_{#1}}
\newc{\up}[1]{{}^{#1}}

\newc{\JulyStrut}{\rule{0mm}{6mm}}
\newc{\midtenPan}{\mbox{\sf S}}
\newc{\midten}{\mbox{\sf T}}
\newc{\midtenEi}{\mbox{\sf U}}
\newc{\ATen}{\mbox{\sf E}}
\newc{\BTen}{\mbox{\sf F}}
\newc{\CTen}{\mbox{\sf G}}

\def\sideremark#1{\ifvmode\leavevmode\fi\vadjust{\vbox to0pt{\vss
 \hbox to 0pt{\hskip\hsize\hskip1em
 \vbox{\hsize2cm\tiny\raggedright\pretolerance10000
  \noindent #1\hfill}\hss}\vbox to8pt{\vfil}\vss}}}

\numberwithin{equation}{section}

\newcommand{\hh}{{\hspace{.3mm}}}

\newcommand{\cc}{\boldsymbol{c}}

\newcommand{\sss}{\scriptscriptstyle}

\newcommand{\ltots}[1]{{\rm ltots}_{#1}}
\newcommand{\APE}[1]{{\rm APE}_{#1}}

\renewcommand\geq{\geqslant}
\renewcommand\leq{\leqslant}

\stackMath
\newcommand\reallywidehat[1]{%
\savestack{\tmpbox}{\stretchto{%
  \scaleto{%
    \scalerel*[\widthof{\ensuremath{#1}}]{\kern-.6pt\bigwedge\kern-.6pt}%
    {\rule[-\textheight/2]{1ex}{\textheight}}
  }{\textheight}%
}{0.5ex}}%
\stackon[1pt]{#1}{\tmpbox}%
}

\definecolor{ao}{rgb}{0.0,0.0,1.0}
\definecolor{forest}{rgb}{0.0,0.3,0.0}
\definecolor{red}{rgb}{0.8, 0.0, 0.0}

\begin{document}
\subjclass[2010]{53C18, 53A55, 53C21, 58J32}

%
%
%


\renewcommand{\today}{}
\title{
{
 A Sharp Characterization of the Willmore Invariant 
}}
%
%
%

\author{ Samuel Blitz${}^\flat$}

\address{${}^\flat$
  Center for Quantum Mathematics and Physics (QMAP)\\
  Department  of Physics\\ 
  University of California\\
  Davis, CA95616, USA} 
   \email{shblitz@ucdavis.edu}
\vspace{10pt}

\renewcommand{\arraystretch}{1}

\begin{abstract}
First introduced to describe surfaces embedded in $\mathbb{R}^3$, the Willmore invariant is a conformally-invariant extrinsic scalar curvature of a surface that vanishes when the surface minimizes bending and stretching. Both this invariant and its higher dimensional analogs appear frequently in the study of conformal geometric systems. To that end, we provide a characterization of the Willmore invariant in general dimensions. In particular, we provide a sharp sufficient condition for the vanishing of the Willmore invariant and show that in even dimensions it can be described fully using conformal fundamental forms and one additional tensor.

\medskip

\noindent
\begin{center}
{\sf \tiny Keywords: 
Extrinsic conformal geometry, hypersurface embeddings, Willmore invariant}
\end{center}

\end{abstract}

\maketitle

\pagestyle{myheadings} \markboth{Blitz}{Higher Forms}

\thispagestyle{empty}

\newpage

\tableofcontents

\newcommand{\balpha}{{\bm \alpha}}
\newcommand{\balphas}{{\scalebox{.76}{${\bm \alpha}$}}}
\newcommand{\bnu}{{\bm \nu}}
\newcommand{\blambda}{{\bm \lambda}}
\newcommand{\bnus}{{\scalebox{.76}{${\bm \nu}$}}}
\newcommand{\bnuss}{\hh\hh\!{\scalebox{.56}{${\bm \nu}$}}}

\newcommand{\bmu}{{\bm \mu}}
\newcommand{\bmus}{{\scalebox{.76}{${\bm \mu}$}}}
\newcommand{\bmuss}{\hh\hh\!{\scalebox{.56}{${\bm \mu}$}}}

\newcommand{\btau}{{\bm \tau}}
\newcommand{\btaus}{{\scalebox{.76}{${\bm \tau}$}}}
\newcommand{\btauss}{\hh\hh\!{\scalebox{.56}{${\bm \tau}$}}}

\newcommand{\bsigma}{{\bm \sigma}}
\newcommand{\bsigmas}{{{\scalebox{.8}{${\bm \sigma}$}}}}
\newcommand{\bbeta}{{\bm \beta}}
\newcommand{\bbetas}{{\scalebox{.65}{${\bm \beta}$}}}

\renewcommand{\bS}{{\bm {\mathcal S}}}
\newcommand{\bB}{{\bm {\mathcal B}}}
\renewcommand{\bT}{{\bm {\mathcal T}}}
\newcommand{\bM}{{\bm {\mathcal M}}}

\newcommand{\go}{{\mathring{g}}}
\newcommand{\nuo}{{\mathring{\nu}}}
\newcommand{\alphao}{{\mathring{\alpha}}}

\newcommand{\Ell}{\mathscr{L}}
\newcommand{\density}[1]{[g\, ;\, #1]}

\renewcommand{\Dot}{{\scalebox{2}{$\cdot$}}}

\newcommand{\PanE}{P_{4}^{\sss\Sigma\hookrightarrow M}}
\newcommand\eqSig{ \mathrel{\overset{\makebox[0pt]{\mbox{\normalfont\tiny\sffamily $\Sigma$}}}{=}} }
\newcommand\neqSig{ \mathrel{\overset{\makebox[0pt]{\mbox{\normalfont\tiny\sffamily $\Sigma$}}}{\neq}} }
\newcommand\eqtau{\mathrel{\overset{\makebox[0pt]{\mbox{\normalfont\tiny\sffamily $\tau$}}}{=}}}
\newcommand{\hd }{\hat{D}}
\newcommand{\hdb}{\hat{\bar{D}}}
\newcommand{\Two}{{{{\bf\rm I\hspace{-.2mm} I}}{\hspace{.2mm}}}{}}
\newcommand{\TwoN}{{\mathring{{\bf\rm I\hspace{-.2mm} I}}{\hspace{.2mm}}}{}}
\newcommand{\Fn}{\mathring{\mathcal{F}}}
\newcommand{\csdot}{\hspace{-0.55mm} \cdot \hspace{-0.55mm}}
\newcommand{\IdD}{(I \csdot \hd)}
\newcommand{\Kd}{\dot{K}}
\newcommand{\Kdd}{\ddot{K}}
\newcommand{\Kddd}{\dddot{K}}

\newcommand{\pdot}{{\boldsymbol{\cdot}}}

 \newcommand{\bdot }{\mathop{\lower0.33ex\hbox{\LARGE$\cdot$}}}


\section{Introduction}
A particularly well-studied equation describing a surface $S$ embedded in flat $\mathbb{R}^3$ is known as the Willmore~\cite{Willmore} equation:
$$ \mathcal{B}_3 := \bar{\Delta} H + 2H (H^2 - k) = 0\,,$$
where $k$ is the Gauss curvature of $S$, $H$ is its mean curvature, $\bar{\Delta}$ is the Laplacian on $S$, and the curvature quantity $\mathcal{B}_3$ is known as the \textit{Willmore invariant}. In particular, one can view the equation $\mathcal{B} = 0$ as the energy minimizing equation (found by functional differentiation) for a surface $S$ with energy given by the Willmore energy,
$$ E = \int_S (H^2 - k) \,, $$
first described in Willmore's celebrated conjecture. This conjecture was only recently proven by Marques and Neves~\cite{WillConj}.

The Willmore energy and invariant can be extended to generally curved ambient manifolds, whereby we observe two critical facts: the Willmore energy and invariant are global and local conformal invariants, respectively, of the embedded surface, and the Willmore invariant is linear in its leading term. Their conformal invariance and linearity are likely major driving factors in the frequent appearance of this pair in both the mathematical and physics literature. Early examples of this include Polyakov~\cite{Polyakov}, who found the Willmore energy when studying strings (subsequently calling it the ``rigid string action'') and Graham and Witten~\cite{GrWi}, who found the Willmore energy when studying anomalies in the AdS/CFT correspondence. Recently, there has been increased interest in higher-dimensional analogs of the Willmore energy and invariant~\cite{Guven,Vyatkin,GR,AGW,Will1}, in part to better understand the invariants of hypersurfaces embedded in conformal manifolds.

A generalization of the Willmore invariant $\mathcal{B}_d$ for hypersurfaces embedded in generally curved $d$-manifolds was introduced by Gover and Waldron~\cite{Will1} using a method that reproduces the classical Willmore functional in $d =3$ and, in $d>3$ odd, takes the form 
$$\mathcal{B}_d := \bar{\Delta}^{\frac{d-1}{2}} H + \text{lower order terms}\,.$$
This generalization maintains the conformal invariance and linearity as desired, and further generalizes to even dimensional ambient spaces (in which linearity is lost). While $\mathcal{B}_3$, $\mathcal{B}_4$, and $\mathcal{B}_5$~\cite{hypersurface_old,BGW2} have been explicitly computed, only some features of $\mathcal{B}_d$ are known for arbitrary $d$. In this paper we seek to make statements about the generalized Willmore invariant in arbitrary dimensions and to thereby contribute to the characterization of conformal hypersurface invariants.

A profitable way to study such invariants is through the method known as \textit{holography}. Given an embedded hypersurface $\Sigma \hookrightarrow M$, one can study invariants of $\Sigma$ by solving a PDE problem on $M$ with data fixed by $\Sigma$. One can then study invariants of $\Sigma$ by investigating the properties of solutions to the PDE problem in $M$.
It is in this context that we will use holographic results to investigate the Willmore functional.

The type of system under investigation in this article is a codimension-$1$ submanifold (a hypersurface) $\Sigma$ conformally embedded (smoothly) in a $d$-dimensional conformal manifold:~${\Sigma \hookrightarrow (M,\cc)}$. A conformal manifold $(M,\cc)$ is a manifold equipped with an equivalence class of metrics $\cc$ such that $[g] = [\tilde{g}]$ if and only if there exists a smooth positive function  $\Omega \in C^\infty M$ such that $\tilde{g} = \Omega^2 g$.  For the purposes of this article, we will assume that $\Sigma$ is closed and orientable, and further that its embedding in $M$ is separating so that $M = M^+ \sqcup \Sigma \sqcup M^-$. Typically, one characterizes a hypersurface by a \textit{defining function} $s \in C^\infty M$ whose zero locus is $\Sigma$ and that satisfies $ds|_p \neq 0$ for every point $p \in \Sigma$. Because $M$ is a conformal manifold, to each $g \in \cc$, we will associate such a defining function $s$: if $g,\tilde{g} \in \cc$ and $\tilde{g} = \Omega^2 g$ for some positive $\Omega \in C^\infty M$ and $s$ is the defining function associated to $g$, then to $\tilde{g}$ we associate the defining function $\Omega s$. We will call such a set of associations a \textit{defining density} $\sigma = [g ; s] = [\Omega^2 g; \Omega s]$---this notion will be elaborated on in more detail in Section~\ref{conf-calc}. For algebraic reasons, we also define the \textit{conformal metric} tautologically by the equivalence class $\gamma = [g; g] = [\Omega^2 g; \Omega^2 g]$. Thus, the data of a hypersurface embedded in a conformal manifold can be summarized by the triple $(M, \gamma, \sigma)$. It is for this triple that we study the Willmore invariant.

A particular family of such triples $(M, \gamma, \sigma)$ has a nice property: they satisfy the \textit{Poincar\'e-Einstein condition}. We say that a conformally embedded hypersurface given by $(M, \gamma, \sigma)$ is Poincar\'e-Einstein (PE) when the Ricci curvature of the singular metric~$g^o := \gamma/\sigma^2$ on $M\backslash \Sigma$ is pure trace, \textit{i.e.} for some negative constant $k$, the singular metric satisfies
$$\Ric^{g^o} = k g^o \text{ on } M\backslash \Sigma\,.$$
Conformal hypersurface embeddings with this property are well-studied. Early work on such embeddings were studied by LeBrun~\cite{LeBrun} and Fefferman and Graham~\cite{FG}. A larger related family of conformal hypersurface embeddings are those that are \textit{asymptotically Poincar\'e-Einstein} (APE) of order $m$. These are triples $(M, \gamma, \sigma)$ that have a singular metric $g^o$ as above that is only asymptotically (in $\sigma$) pure trace, \textit{i.e.}
$$\Ric^{g^o} = k g^o + \mathcal{O}(\sigma^m) \text{ on } M\backslash \Sigma\,.$$
There are quite severe restrictions on conformal hypersurface embeddings that belong to either of these families~\cite{BGW}. Indeed, the vast majority of such embeddings do not belong to either family. Schematically, we will write that $(M, \gamma, \sigma) \in \text{PE}$ when $(M, \gamma, \sigma)$ satisfies the PE property, and similarly $(M, \gamma, \sigma) \in \APE{m}$ when the triple satisfies the APE property of order $m$.

We call hypersurface embeddings that have a vanishing Willmore invariant $\mathcal{B}_d$ \textit{Willmore} and use the notation $(M^d, \gamma, \sigma) \in \mathcal{W}_d$ to denote such a hypersurface embedding. A result of~\cite{Goal} shows that $\text{PE} \subset W_d$. Because $\text{PE} \subset \APE{k}$, it is of interest to investigate the inclusion properties of $\APE{}$ in $\mathcal{W}_d$. In this article we provide an exact lower bound on the order of an APE triple $(M, \gamma, \sigma)$ such that the embedding is necessarily Willmore.

\begin{theorem} \label{ape-B}
For conformal hypersurface embeddings specified by $(M^d,\gamma,\sigma)$ with $d \geq 4$, we have that $ \APE{d-3} \subset \mathcal{W}_d$ but $\subset \APE{d-4} \not \subset \mathcal{W}_d $.
\end{theorem}
The proof of this result is given in Section~\ref{results}. A consequence of the proof of this theorem is that, for $d$ even, the Willmore invariant can always be written in terms of a specific finite set of fundamental tensors; this result is captured in Corollary~\ref{B-makeup}. These tensors, initially characterized in~\cite{BGW}, are described in more detail in Section~\ref{cff-section}. A critical tool used to prove this corollary is known as the \textit{ambient construction} and is summarized in Section~\ref{proofs}, along with other technical results required for the proof of Corollary~\ref{B-makeup}.

\subsection{Riemannian notations}
Here we provide a summary of the notations we will use to describe Riemannian geometry. The manifold $M$ will always be given the dimension $d$ and be assumed to be smooth. Further, this manifold will be endowed either with a Riemannian metric $g$ or a conformal class of metrics $\cc$, described in Section~\ref{conf-calc}. Given a Riemannian metric $g$, we will denote the Levi-Civita connection on $(M,g)$ with $\nabla$ and its corresponding curvature by
$$R(x,y)z = (\nabla_x \nabla_y - \nabla_y \nabla_x) z - \nabla_{[x,y]} z$$
where $x,y,z$ are smooth vector fields on $M$ and $[\pdot, \pdot]$ is the Lie bracket.

For the purposes of this article, we will use Penrose's abstract index notation to represent tensors. This notation will also be used to represent tensor contractions via the Einstein index summation convention. For example, we may write the above description of the Riemann curvature as
$$x^a y^b R_{ab}{}^c{}_d z^d = x^a \nabla_a y^b \nabla_b z^c  - y^a \nabla_a x^b \nabla_b z^c - (x^a (\nabla_a y^b) - y^a (\nabla_a x^b)) \nabla_b z^c\,.$$
Observe that in the above formula we assume that Levi-Civita connections act on everything to their right unless they are enclosed by brackets. Alternatively, we will sometimes place (co)vectors in the place of the abstract indices to represent the appropriate contraction so that~$R_{xy}{}^c{}_z \equiv x^a y^b R_{ab}{}^c{}_d z^d$.
Using this notation, we can use the isomorphism between $TM$ and $T^*M$ via the metric $g_{ab}$ and its inverse $g^{ab}$ to ``raise'' and ``lower'' indices, \textit{i.e.} $v_a = g_{ab} v^b$. Similarly, the Kronecker delta $\delta_a^b$ is used to represent the identity isomorphism. When the contraction structure is clear, we will often use a dot product-like notation to represent a contraction of indices. For example, the action of a $1$-form $u_a$ on a vector $v^b$ can be written as $u \csdot v \equiv u_a v^a$. This same notation can be used for more complex tensor structures so that we may write $u \csdot v \equiv u_{abc\dots} v^{abc \dots}$. Similarly, for inner products with the metric or its inverse, we will use the same notation. For two vector fields $u,v$, we will write $u \csdot v = g_{ab} u^a v^b$, and similarly for the inner product on the cotangent bundle.

Additionally, we will use round brackets to represent the symmetric part of a given tensor structure so that, for example, $v_{a(bc)d} := \tfrac{1}{2} (v_{abcd} + v_{acbd})$, and square brackets to represent the antisymmetric part, \textit{i.e.} $v_{a[bc]d} := \tfrac{1}{2} (v_{abcd} - v_{acbd})$. When we are only interested in the trace-free part of some symmetric part of a tensor, we will follow the round brackets with a $\circ$ so that $v_{(ab)\circ} = v_{(ab)} - \tfrac{1}{d} g_{ab} v^c{}_c$. Alternatively, when the tensor is fully symmetric and trace-free, we will sometimes place a $\circ$ over the tensor so that $v_{(ab)\circ} \equiv \mathring{v}_{(ab)}$. Symmetric and antisymmetric tensor products of vector bundles will be represented by $\odot$ and $\wedge$, respectively, and the trace-free symmetric tensor products of vector bundles will be represented by $\odot_\circ$. While the above discussion was given in the context of the tangent and cotangent bundles of $M$, the same language will be used (with different index names) for different vector bundles.

We will often refer to an arbitrary but fixed tensor product of a given vector bundle $\mathcal{V} M$ by $\mathcal{V}^\Phi M$, where in general we will use capital Greek letters to represent such arbitrary but fixed tensor structures. In particular, if $t \in \Gamma(\mathcal{V}^\Phi M)$, then $t$ is a tensor of type $\Phi$. For the standard (co)tangent bundles we will instead use lowercase Greek letters so that for $t \in \Gamma(T^\phi M)$, we have that $t$ is of tensor type $\phi$.

In these notations, we can decompose the Riemann curvature tensor into the Weyl tensor and the Schouten tensor according to
$$R_{abcd} = W_{abcd} + g_{ac} P_{bd} - g_{bc} P_{ad} - g_{ad} P_{bc} + g_{bd} P_{ac}\,.$$
Here the Schouten tensor is a trace-correction of the Ricci tensor $Ric_{ab} := R_{ca}{}^c{}_b$ given, in dimensions $d \geq 3$, by the formula
$$P = \tfrac{1}{d-2} \left(Ric - \frac{R}{2(d-1)} g\right)$$
where $R := Ric_a^a$. We denote by $J := P_a^a$ the trace of the Schouten tensor. A useful fact is that the covariant curl of the Schouten tensor and the divergence of the Weyl tensor are proportional to one another; in particular, defining the \textit{Cotton tensor} by
$$C_{abc} := \nabla_{[a} P_{b]c}\,,$$
the Cotton tensor satisfies for $d \geq 4$ the relationship $(d-3) C_{abc} = \nabla^d W_{dcab}$. When the metric associated to any curvature tensor or differential operator is not clear, we will use a superscript $g$ to indicate this.

\section{Background}

Much of what follows in this section is a brief summary of explanations and results given in~\cite{Will1, Will2, BGW}.

\subsection{Conformal calculus} \label{conf-calc}

As mentioned in the introduction, on a conformal manifold $(M, \cc)$ there exist equivalence classes (called ``conformal densities'') associated to the conformal class of metrics $\cc$. These objects are ``conformally covariant'' in the sense that they transform predictably under conformal rescalings of the metric; however, we will often refer to such objects simply as conformally invariant. Naively, a \textit{conformal density of weight $w$}  is a double equivalence class $\phi = [g; f] = [\Omega^2 g; \Omega^w f]$ where $f \in C^\infty M$, and we denote the bundle to which weight $w$ densities belong $\Gamma(\ce M[w])$. Here, $\Gamma$ will denote the section space of a bundle and $\ce M[w]$ is the weight-$w$ line bundle.  It is also useful to consider tensor-valued densities: a tensor-valued (of tensor type $\Theta$) weight $w$ density $\theta = [g; t]$ belongs to the section space of the product bundle $T^\Theta M \otimes \ce M[w] := T^\Theta M[w]$; this notation is generically used to refer to vector bundle-valued densities. A fundamental example of such a density is the conformal metric: $\gamma \in \Gamma(\odot^2 T^* M[2])$. Of special interest are densities of weight $1$; these are referred to as \textit{scales}. In particular, a nowhere vanishing scale $\tau$ is called a \textit{true scale} and canonically determines a metric representative from the conformal class $g_\tau \in \cc$ by trivializing the conformal metric via $g_\tau := \gamma/\tau^2$.

A Riemannian metric is not uniquely given for a conformal manifold, so there is no uniquely defined connection on $(M,\cc)$. Instead, there exists a family of Levi-Civita connections: for each $g_\tau \in \cc$, the Levi-Civita connection acting on weight $w$ densities can be written as~$\nabla^\tau := \tau^{w} \circ \nabla^{g_\tau} \circ \tau^{-w}$, where $\nabla^{g_\tau}$ is the usual Levi-Civita connection on $(M, g_\tau)$. For notational brevity, we will always drop the superscript $g_\tau$ when it is implied. For a more geometric description of these structures, see~\cite{GW}.

In addition to the above structures, there exists a natural calculus on conformal manifolds called the tractor calculus that allows one to systematically construct conformally invariant tensors and operators~\cite{BEG}. The tractor calculus consists of a rank $d+2$ vector bundle, known as the \textit{standard tractor bundle} and a canonical connection on this bundle. Given a conformal manifold $(M, \cc)$, for any given metric representative $g \in \cc$, there exists an isomorphism between the standard tractor bundle and a direct sum of density bundles:
$$
\ct M \stackrel{g}{\cong} \ce M[1]\oplus TM[-1]\oplus \ce M[-1]\, .
$$
When such a \textit{choice of splitting} is made (also sometimes called a \textit{choice of scale}), we can decompose a section of the standard tractor bundle $T \in \Gamma(\ct M)$ into the ordered triple,
$$T^A \stackrel{g}= (\tau^+, \tau^a, \tau^-)\,,$$
or the corresponding column vector notation. The relationship between the isomorphisms given two conformally related metrics $g$ and $\Omega^2 g$ implies that
$$T^A \stackrel{\Omega^2 g}= (\tau^+, \tau^a + \Upsilon^a \tau^+, \tau^- - \Upsilon \csdot \tau - \tfrac{1}{2} |\Upsilon|^2_g \tau^+)\,,$$
where $\Upsilon = d (\ln \Omega)$.
Often we will use the language of\;\;``slots'' to refer to specific entries in a tractor viewed as a vector or matrix. In particular, in the above display, the entry containing $\tau^+$ might be referred to as the ``top slot'' (where the column vector notation is used for visualization). For the square of the tractor bundle, one might refer to the component corresponding to the term $TM[-1] \otimes TM[-1]$ in a choice of splitting as the ``middle slot.'' Such language will be used sparingly throughout this article.

Just as (co)tangent tensor bundles are given weights (by taking the product with a density bundle) so too can tractor bundles: we will write $\ct M[w] := \ct M \otimes \ce M[w]$ and we will also call these bundles \textit{tractor bundles}. Using the above display, we can observe a key feature: the first non-zero entry in a tractor defines a density independent of the choice of splitting, which implies that said quantity is conformally invariant. In any given tractor, this first non-zero slot is referred to as the \textit{projecting part} of the tractor; the map $q^*$ extracts this term and is called the \textit{extraction map}. Additionally, observe that the section $X \stackrel{g}= (0,0,1) \in \Gamma(\ct M[1])$ is also canonically defined independently of the choice of splitting and hence is called the \textit{canonical tractor}. Finally, there exists a natural symmetric, non-degenerate (but non-positive) inner product between tractors that is independent of the choice of splitting, given by
$$h(U,V) \stackrel{g}= u^+ v^- + \gamma_{ab} u^a v^b + u^- v^+\,,$$
for two tractors $U,V \in \Gamma(\ct M)$ specified in a choice of splitting in the obvious way.
This inner product defines the \textit{tractor metric} $h_{AB} \in \Gamma(\odot^2 \ct M)$ given, in a choice of splitting, by
$$
h_{AB}\stackrel{g}=\begin{pmatrix}
0&0&1\\
0& \gamma_{ab}&0\\
1&0&0
\end{pmatrix}\, ,
$$
with its inverse denoted by $h^{AB}$. Just like the Riemmanian metric for Riemmanian tensors, the tractor metric provides an isomorphism between the tractor bundle and its dual
$$\ct^* M \stackrel{g}{\cong} \ce M[1] \oplus T^M[1] \oplus \ce M[-1]\,.$$

Given the tractor metric, there exists a canonical \textit{tractor connection} on tractor bundles $$\nabla^\ct : \Gamma(\ct M) \rightarrow \Gamma(\ct M \otimes T^* M)$$ given in a choice of splitting by 
\begin{equation} \label{trac-conn}
\nabla_a^\ct T^B \stackrel{g}= (\nabla_a \tau^+ - \tau_a, \nabla_a \tau^b + \delta_a^b \tau^- + (P^g)_a^b \tau^+, \nabla_a \tau^- - P_{ab}^{g} \tau^b)\,,
\end{equation}
where $\nabla$ is the density-coupled Levi-Civita connection and $P^g$ is the Schouten tensor associated with the splitting metric $g$. This connection extends as usual to arbitrary tensor products of the tractor bundle as well as products of the standard tractor bundle with density bundles. Just as for the density-coupled Levi-Civita connection, we will usually drop the superscript $\ct$ in $\nabla^\ct$ when the connection is clear from context.

Observe that the action of the tractor connection on a tractor is not unique but depends on a choice of scale (used to determine the density-coupled connection). However, there exists an invariant second order operator containing the tractor connection that does have this property: the \textit{Thomas-$D$ operator}~\cite{Thomas}. Given a tractor $T \in \Gamma(\ct^\Phi M[w])$, we write in a choice of splitting
\begin{equation} \label{D-trac}
\resizebox{.9\hsize}{!}{$
D^A T^\Phi \stackrel{g}= \left( w(d+2w-2) T^\Phi, (d+2w-2) \nabla^\ct T^\Phi, -\Delta^\ct T - wJ^g T \right) \in \Gamma(\ct M \otimes \ct^\Phi M[w-1])\,,
$}
\end{equation}
where $\Delta^\ct := \gamma^{ab} \nabla^\ct_a \nabla^\ct_b$ is the tractor Laplacian and $J^g$ is the trace of the Schouten tensor associated to the splitting metric $g$. A useful modification of this operator is the ``hatted'' Thomas-$D$ operator $\hd$. For any tractor $T \in \Gamma(\ct^\Phi M[w])$ where $w \neq 1 - \frac{d}{2}$, $\hd$ is defined by
$$\hd^A T^\Phi := (d+2w-2)^{-1} D^A T^\Phi\,.$$
Importantly, the Thomas-$D$ operator is not a derivation because it is a second-order differential operator; however, it does have the property that $D^A D_A = 0$. Further, the operator $\hd$ acting on a product of tractors satisfies a relationship known as the \textit{Leibniz failure}~\cite{Will1}, recorded here for our use.
\begin{proposition}[Leibniz failure]
Let $T_i \in \Gamma(\ct^\Phi_i M[w_i])$ for $i = 1,2$, $h_i := d + 2w_i$, and $h_{12} := d+2w_1 + 2w_2 $ such that $h_i \neq 2 \neq h_{12}$. Then,
$$\hd^A(T_1 T_2) - (\hd^A T_1) T_2 - T_1 (\hd^A T_2) = -\tfrac{2}{h_{12}} X^A (\hd^B T_1) (\hd_B T_2)\,.$$
\end{proposition}

While in principle, this is enough to use the tractor calculus, there exists one more canonical tractor that roughly plays the role of curvature for the Thomas-$D$ operator: the $W$-tractor. Acting on a weight $w$ tractor $T \in \Gamma(\ct M[w])$ with $w \neq 1-\frac{d}{2}, 2 - \frac{d}{2}$, the commutator of Thomas-$D$ operators obey~\cite{GOmin,BGW}
\begin{equation}\label{W-trac}
[\hd_A, \hd_B] T^D = W_{AB}{}^D{}_E  T^E + \tfrac{4}{d+2w-4} X_{[A} W_{B]C}{}^D{}_E \hd^C T^E\,,
\end{equation}
extending as a derivation to other tensor structures. Note that the projecting part of the $W$ tractor is the Weyl tensor. Further, observe that $X \csdot W = 0 = \hd \csdot W$.

The canonical tractor $X$, the tractor metric $h$ (and its inverse), the operators $D$ and $\hd$, and the $W$-tractor comprise the key ingredients of the tractor calculus; many invariants of a given conformal manifold can be constructed and extracted from these objects when  combined with other data specifying a system of interest.

In order to construct particular conformally-invariant tensors from tractors, one often wishes to combine the extraction operator $q^*$ mentioned above with an operator that changes the projecting part of a tractor in a well-defined way. Such an operator was defined in~\cite[Lemma 2.11]{BGW}, which is recorded here for our later use.
\begin{lemma} \label{r-operator}
Let $T \in \Gamma(\odot^2_\circ \ct M[w])$ such that $w \neq 0, -1, -\tfrac{d}{2}, -1 - \tfrac{d}{2}, -2-\tfrac{d}{2}$. Then there exists an operator
$$r : \Gamma(\odot^2_\circ \ct M[w]) \rightarrow \Gamma(\odot^2_\circ \ct M[w])$$
defined by
\begin{multline*}
r(T^{(AB)\circ}) := T^{(AB)\circ} - \tfrac{2}{w} \hd^{(A} (X_C T^{|C|B)\circ}) + \tfrac{1}{w(w+1)} \hd^{(A} \hd^{B)\circ} (X_C X_D T^{CD}) \\-\tfrac{8}{wd(d+2w+2)} X^{(A} \hd^{B)\circ} \hd^C (X_D T^{CD})
\end{multline*}
that obeys
$$X_A r(T^{AB}) = 0 = h_{AB} r(T^{AB})\,.$$
\end{lemma}
The operator $r$ adds tractor-valued terms to a trace-free rank-2 symmetric tractor $T$ so that the projecting part is the middle slot rather than the top slot. It can then be composed with the extraction operator to extract a conformally invariant rank-2 symmetric trace-free tensor. That is, for a generic $T \in \Gamma(\odot^2_\circ \ct M[w])$, we have that
$$(q^* \circ r)(T^{AB}) \in \Gamma(\odot^2_\circ TM[w-2]) \cong \Gamma(\odot^2_\circ T^*M[w+2])\,,$$
where the isomorphism is obtained in the usual way with the conformal metric $\gamma$.

\subsection{Conformal fundamental forms} \label{cff-section}
The discussion in this section summarizes the results of~\cite{BGW}.

Given a conformally embedded hypersurface $(M, \gamma, \sigma)$, there exist a family of conformally invariant (transverse) jet coefficients for the defining density $\sigma$. The first such jet coefficient is the conormal $n := d \sigma|_{\Sigma} \in \Gamma(T^* M[1])|_{\Sigma}$. The \textit{unit} conormal is particular useful, so we typically consider $\hat{n} := n/|n|_\gamma$. The unit conormal allows for a decomposition of the ambient tangent bundle into the hypersurface tangent bundle and the normal bundle by the hypersurface projector (also sometimes called the first fundamental form)
$$\bar{\gamma}_a^b := \delta_a^b - \hat{n}_a \hat{n}^b\,.$$
For notational simplicity, because there is such an isomorphism, we use the same abstract index notation to represent tensors on the tangent and normal bundles on $\Sigma$ as on the tangent bundle on $M$. Note that the induced conformal metric on $\Sigma$, denoted by $\bar{\gamma}$, can be computed by application of the projector to the ambient conformal metric: $\bar{\gamma}_{ab} := \bar{\gamma}_a^{a'} \bar{\gamma}_b^{b'} \gamma_{a'b'} = \gamma_{ab} - \hat{n}_a \hat{n}_b \in \Gamma(\odot^2 T^* \Sigma[2])$. Going forward, we use bars over symbols to indicate that the symbol belongs to the hypersurface; for example, we will write $\bar{\nabla}$ for the Levi-Civita connection on $\Sigma$.

The next conformally invariant jet of $\sigma$ that we are interested in is the (trace-free) second fundamental form. Given $\bar{g} \in \cc$, the second fundamental form is given by $\II := \bar{g}_a^c \nabla_c \hat{n}_b^{\rm e}|_{\Sigma}$ where $\hat{n}^{\rm e}$ is any extension of the unit conormal and we denote the projection of the ambient Levi-Civita connection by $\nabla^\top_a := \bar{g}_a^b \nabla_b$. Throughout this article we will use often write $t^\top$ to denote the projection of $t$ to $\Sigma$ for $t$ any tensor or tensor-valued operator. Additionally, $\top$ will sometimes be treated as the projection operator and $\otop$ will refer to the trace-free part of the operator $\top$. For a vector field $x \in \Gamma(TM)$ and $T$ some tensor field on $M$, sometimes the notation $T_{xabc\ldots}^\top$ will be used to represent $\top (x^{a'} T_{a'abc\ldots})$.

Observe that the second fundamental form is not conformally invariant because its trace $(d-1)H$ (where $H^g$ is the mean curvature of $\Sigma \hookrightarrow (M,g)$) is not density valued. Thus, only the trace-free part of the second fundamental form is conformally-invariant: $\IIo \in \Gamma(\odot^2_\circ T^* \Sigma[1])$. For notational simplicity, we will often refer to a representative of a density using the same symbol as the density itself.

Studying higher transverse order jet coefficients requires more finesse: indeed, we require a conformally-invariant extension of $\IIo$. For the purposes of the calculations to be carried out in the remainder of this article, we use the canonical choice~\cite{BGW}
$$\IIo^{\rm e} = [g; \nabla_{(a} \nabla_{b)\circ} s + s P_{(ab)\circ}] \in \Gamma(\odot^2_\circ T^* M[1])\,.$$
A simple calculation verifies that $\IIo^{\rm e}|_{\Sigma} = \IIo$. Observe that this choice of extension is canonical because it is the projecting part of $\nabla_a \hd_B \sigma$ when, in a choice of scale given by $\sigma = [g;s]$, we have that $|ds|_g^2 \eqSig 1$. (We will often use ``$\eqSig$'' to denote equality that holds along $\Sigma$.) Computing the conformally-invariant higher-order transverse jet coefficients amounts to applying conformally-invariant normal derivative operators of sequentially higher transverse order to $\IIo^{\rm e}$. This procedure was carried out in detail in~\cite{BGW}, where such jet coefficients were called \textit{(conformal) fundamental forms}. In particular, an $m$th conformal fundamental form is a symmetric rank-2 tensor-valued hypersurface density with weight $3-m$ and transverse order $m-1$. 

The notion of the \textit{transverse order} of a tensor is defined using an appropriate notion of derivative counting of that tensor's \textit{preinvariant}, defined in~\cite{Will1}: roughly, one writes the tensor in question in terms of a metric $g$ and partial derivatives with respect to the transverse coordinate $s$ and tangential coordinates $y$ and then computes the maximal number of partial derivatives $\partial_s$ that act on $g$ when restricted to $\Sigma$. We say that the transverse order of an operator ${\sf O}$ is $k$ when there exists $v$ in the domain of ${\sf O} \circ s^k$ such that $O(s^k v)|_{\Sigma} \neq 0$ but for all $v'$ in the domain of ${\sf O} \circ s^{k+1}$, ${\sf O}(s^{k+1} v') \eqSig 0$. 

A general construction of conformal fundamental forms was given using an iterative procedure. The first two non-trivial examples using this canonical construction are the third and fourth fundamental forms denoted by $\IIIo$ and $\IVo$ respectively. For $d \geq 4$,
$$\IIIo_{ab} := [g; -\IIo^2_{(ab)\circ} + W_{\hat n ab \hat n}] \in \Gamma(\odot^2_\circ T^* \Sigma[0])$$
and for $d \geq 6$, 
\begin{align*}
 \IVo_{ab} :=
 [g; &-(d-4)(d-5) C_{\hat n(ab)}^\top  - (d-4)(d-5) H W_{\hat n ab \hat n} - (d-4) \bar \nabla^c W_{c(ab) \hat n}^\top \\&+ 2 W_{c \hat n \hat n (a}^{} \IIo_{b)\circ}^c - \tfrac{d^2 - 7d+18}{d-3} \IIIo_{(a} \csdot \IIo_{b)\circ} + (d-6) \bar{W}^c{}_{ab}{}^d \IIo_{cd} \\&+ \tfrac{d^3-10d^2+25d-10}{(d-1)(d-2)} K \IIo_{ab}] \in \Gamma(\odot^2_\circ T^* \Sigma[-1])
 \, ,
\end{align*}
where $K := \IIo^2$. Note that these tensor-valued densities are independent of the choice of defining density $\sigma$ for $\Sigma$. In their construction, there is a subtlety involving the dimensionality of the manifold $M$. Indeed, for $2 \leq m < \frac{d+3}{2}$, the $m$th canonical fundamental form $\FF{m}$ is always defined. However, for $\frac{d+3}{2} \leq m \leq d-1$, in general such fundamental forms do not exist; rather, such tensors are only well-defined when an $n$th fundamental form vanishes for each $2 \leq n < \frac{d+3}{2}$.

For the purposes of this article, it is useful to obtain a more specific but less conditional construction of fundamental forms---this construction is carried out in Section~\ref{new-FFs-sec}, with the trade-off that such a construction only holds for $d$ even. However, observe that we can continue to use the same notation for these fundamental forms by a uniqueness result proved in~\cite{BGW}: for any conformal fundamental form, its leading transverse derivative tensor structure is unique up to an overall coefficient.

\subsection{Tractor holography}
As hinted at in the introduction, one way to study the invariants of a conformal hypersurface embedding $\Sigma \hookrightarrow (M,\cc)$ is by enforcing a PDE problem on $(M,\cc)$. Indeed, this is particularly useful because a conformal hypersurface embedding does not have a unique triple $(M,\cc,\sigma)$: for any positive $f \in C^\infty M$, the triple $(M,\cc,f \sigma)$ describes the same system. To canonically choose such a $\sigma$, we demand that $\sigma$ satisfy some natural problem and study the resulting structure using the tractor calculus and the defining density $\sigma$. To that end, we specify the \textit{singular Yamabe problem}~\cite{Yamabe}:
\begin{problem}
Given a conformal hypersurface embedding $(M, \gamma, \sigma)$, find a positive function~$f \in C^\infty M$ such that the singular metric $g^o$ associated with the triple $(M, \gamma, f \sigma)$ has a constant scalar curvature:
$$Sc^{g^o} = -d(d-1)\,.$$
\end{problem}

\noindent
An analogous problem was solved by Loewner and Nirenberg on round structures~\cite{Loewner}.
In general, a one-sided global solution to this problem always exists~\cite{Aviles,Maz,ACF} but depends on global information about $M^+$. On the other hand, a solution $f$ that depends only on local data of the embedding can always be found such that $g^o = \gamma/(f\sigma)^2$ asymptotically solves the singular Yamabe problem~\cite{ACF,Goal}:
$$Sc^{g^o} = -d(d-1) + \mathcal{O}(\sigma^d)\,.$$
In this paper, we will use the notation $\mathcal{O}(s^m)$ (or, for densities, equivalently $\mathcal{O}(\sigma^m)$) to indicate that the remaining terms in an expression can be written as $s^m f$ where $f$ is some function that is regular in the limit where $s$ approaches zero.

This problem can be framed in the language of tractors~\cite{Goal,GW} described above. If $(M, \gamma, \sigma)$ asymptotically solves the singular Yamabe problem, then
$$I_A I^A = 1 + \sigma^d B\,,$$
where $I^A := \hd^A \sigma$ is known as the \textit{scale tractor} and $B \in \Gamma(\ce M[-d])$. The restriction of $B$ to $\Sigma$ is called the \textit{obstruction density}. Because this asymptotic solution always exists and is uniquely determined, we can uniquely (up to order $\sigma^d$) specify a $\sigma$ in a triple $(M,\gamma,\sigma)$ for a conformal hypersurface embedding by demanding that $\sigma$ asymptotically solves the singular Yamabe problem above. We will denote such triples by $(M,\gamma,\sigma)_{\mathcal{Y}}$.  Then, we can study the extrinsic invariants of this hypersurface embedding by utiilizing the tractor calculus together with the defining density $\sigma$. Indeed, one can construct transverse ``tractor jets.'' The \textit{normal tractor} $N \in \mathcal \ct M|_{\Sigma}$, given  in a choice of splitting by $N^A \stackrel{g}= (0,\hat{n}, -H^g)$, is an analog of the unit conormal that satisfies $I^A \eqSig N^A$ so long as $I^2 = 1 + \mathcal{O}(\sigma^2)$. The next such jet coefficient is the tractor $P_{AB} := \hd_A I_B$, which has as its projecting part $\IIo^{\rm e}$; it can be written in a choice of splitting as
\begin{equation} \label{Ptrac}
P_{AB}\stackrel g=
\begin{pmatrix}
0&0&0\\
0&\IIo^{\rm e}_{ab}&
-\tfrac1{d-1} \nabla\csdot \IIo^{\rm e}_a\\[1mm]
0& -\tfrac1{d-1} \nabla\csdot \IIo^{\rm e}_b &
\frac{ \nabla\cdot\nabla \cdot \IIo^{\rm e} + (d-1)\Rho^{ab} \IIo^{\rm e}_{ab}}
{(d-1)(d-2)}
\end{pmatrix}.
\end{equation}
A consequence of the relationship between $P_{AB}$ and $\IIo^{\rm e}$ is the following lemma.
\begin{lemma} \label{nIIoe}
Let $(M,\gamma,\sigma)_{\mathcal{Y}}$ represent a conformally embedded hypersurface and let $ \ell \in \mathbb{Z}_{\leq d-2}$ be non-negative. If $\IIo^{\rm e} = \mathcal{O}(\sigma^{\ell})$, then $n^a \IIo^{\rm e}_{ab} = \mathcal{O}(\sigma^{\ell+1})$ and $\nabla \csdot \IIo^{\rm e} = \mathcal{O}(\sigma^{\ell})$.
\end{lemma}

\begin{proof}
First observe that, when $\IIo^{\rm e} = \mathcal{O}(\sigma^{\ell})$, the extension of the tangential divergence of $\IIo^{\rm e}$ satisfies $\nabla_{\rm e}^{\top a} \IIo^{\rm e}_{ab} := (g^{ac} - n^a n^c) \nabla_c \IIo^{\rm e}_{ab} = \mathcal{O}(\sigma^{\ell})$. This follows because 
$$[\nabla_n, \nabla_{{\rm e}\, a}^{\top}] = 2 \rho n_a\nabla_n + \ltots{0}\,,$$
where $\rho := -\tfrac{1}{d}(\Delta s + Js)$. Note that by $\ltots{k}$ we mean any (possibly differential) operator with transverse order less than or equal to $k$. That is, we can write
\begin{equation} \label{div-n1}
\nabla \csdot \IIo^{\rm e} - n \csdot \nabla_n \IIo^{\rm e} = \nabla^\top_{\rm e} \csdot \IIo^{\rm e} = \mathcal{O}(\sigma^{\ell})\,.
\end{equation}
Further, observe that by direct utilization of the Leibniz failure,
$$\frac{1}{2} \hd_A I^2 = I^B P_{AB} - \frac{K_{\rm e} X_A}{d-2}\,,$$
where $K_{\rm e} := (\IIo^{\rm e})^2 = P_{AB} P^{AB}$. On the other hand, because $I^2 = 1 + \sigma^d B$, we can compute $\tfrac{1}{2} \hd_A I^2$ directly, yielding
$$q^*(2 I^B P_{AB}) = d \sigma^{d-1} n_a B + \sigma^d \nabla_a B\,.$$
From Equation~\nn{Ptrac}, we have that $q^*(I^A P_{AB}) = -\frac{\sigma}{d-1} \nabla \csdot \IIo^{\rm e} + n \csdot \IIo^{\rm e}$. Thus, we have that
\begin{equation} \label{div-n2}
\sigma \nabla \csdot \IIo^{\rm e}- (d-1) n \csdot \IIo^{\rm e}_b =\mathcal{O}(\sigma^{d-1})\,.
\end{equation}
Now choose an integer $0 \leq k \leq d-2$ and observe that $[\nabla^m_n, n_a] = \ltots{m-1}$ for any positive integer $m$. Combining Equations~\ref{div-n1} and~\ref{div-n2} and taking $k$ transverse derivatives, we find that
\begin{equation} \label{dn-delt-IIoe}
\nabla_n^k n \csdot \IIo^{\rm e} \eqSig \tfrac{k}{d-1-k} \nabla_n^{k-1}  \nabla^{\top}_{\rm e} \csdot \IIo^{\rm e} + \ltots{k-1}(n \csdot \IIo^{\rm e}) + \ltots{k-2}(\nabla^\top_{\rm e} \csdot \IIo^{\rm e})\,,
\end{equation}
where for $m < 0$, $\ltots{m}$ is the zero operator. Observe that for $k \leq \ell$, the right hand side vanishes to order $\mathcal{O}(\sigma^{\ell-k+1})$. Thus, if $\IIo^{\rm e} = \mathcal{O}(\sigma^{\ell})$, we have that $\nabla_n^k n \csdot \IIo^{\rm e} \eqSig 0$ and hence $n \csdot \IIo^{\rm e} = \mathcal{O}(\sigma^{\ell+1})$. The remainder of the lemma follows from Equation~\nn{div-n1}.
\end{proof}

\begin{remark} \label{dt-comm}
Because $[\nabla_n, \nabla_{{\rm e}\,a}^\top] = 2 \rho n_a \nabla_n + \ltots{0}$, it follows from Lemma~\ref{nIIoe} that
\begin{align} \label{delt-dn-IIoe}
\nabla_n^{k-1}  \nabla^{\top}_{\rm e} \csdot \IIo^{\rm e} \eqSig \nabla^\top \csdot \nabla_n^{k-1} \IIo^{\rm e} + \ltots{k-2}(\IIo^{\rm e})
\end{align}
for $k \leq d-2$.
\end{remark}

\section{The Willmore invariant}\label{results}

To prove Theorem~\ref{ape-B}, we first need to relate the obstruction density to the Willmore invariant. A key result for our work here is~\cite[Theorem 5.1]{Will1}, recorded here for reference.
\begin{theorem}
Up to a non-zero constant multiple, the obstruction density $B|_{\Sigma}$ takes the form
$$\bar{\Delta}^{\frac{d-1}{2}} H + \text{lower order terms}$$
for $d$ odd and is fully non-linear for $d$ even.
\end{theorem}
With this result in mind, we give the following definition in line with the discussion in the introduction~\cite{Will1}:
\begin{definition}
Let $(M,\gamma,\sigma)_{\mathcal{Y}}$ specify a conformal hypersurface embedding into a $d$-dimensional conformal manifold $(M,\cc)$ such that
$$I_A I^A = 1+ \sigma^d B\,.$$
Then the \textit{generalized Willmore invariant} $\mathcal{B}_d$ is defined by
$$\mathcal{B}_d := B|_{\Sigma} \in \Gamma(\ce \Sigma[-d])\,.$$
\end{definition}

With this definition in place, we now provide an elementary proof of Theorem~\ref{ape-B}.
\begin{proof}[Proof of Theorem~\ref{ape-B}]
To prove this theorem we first show for some conformal hypersurface embedding specified by $(M^d, \gamma, \sigma) \in \APE{d-3}$, we have that $(M^d, \gamma,\sigma) \in \mathcal{W}_d$. From the definition of $\APE{d-3}$, we have that $\mathring{P}^{g^o} = \mathcal{O}(\sigma^{d-3})$. Further, it is easy to verify that $\IIo^{\rm e} = \sigma \mathring{P}^{g^o}$. So it suffices to check that $\mathcal{B}_d = 0$ when $\IIo^{\rm e} = \mathcal{O}(\sigma^{d-2})$.

First, observe that $I^2 = 1 + \sigma^d B$, so in particular $I^B \nabla_n I_B = \frac{d}{2} \sigma^{d-1} B + \mathcal{O}(\sigma^{d})$. Using Equation~\nn{trac-conn} in a choice of scale $\gamma = [g;s]$,  we have that
$$\frac{d}{2} s^{d-1} B + \mathcal{O}(s^{d}) \stackrel{g}= n^a n^b \IIo^{\rm e}_{ab} - \tfrac{1}{d-1} s n^a \nabla^b \IIo^{\rm e}_{ab}\,.$$
Taking $d-1$ normal derivatives and evaluating along $\Sigma$, we have that
\begin{align*}
\frac{d!}{2} B \eqSig&  \nabla_n^{d-1} n^a n^b \IIo^{\rm e}_{ab} - \nabla_n^{d-2} n^a \nabla^b \IIo^{\rm e}_{ab} \\
\eqSig&  \nabla_n^{d-1} n^a n^b \IIo^{\rm e}_{ab} - \nabla_n^{d-2} n^a \nabla^{\top b}_{\rm e} \IIo^{\rm e}_{ab} - \nabla_n^{d-2} n^a n^b \nabla_n \IIo^{\rm e} \\
\eqSig&  \nabla_n^{d-2} \left[2 n^{(a} (\nabla_n n^{b)}) \IIo^{\rm e}_{ab}- n^a \nabla^{\top b}_{\rm e} \IIo^{\rm e}_{ab}  \right]\\
\eqSig& \nabla_n^{d-2} \left[ n^{(a} \IIo^{\rm e}_{ab} \nabla^{b)} (1 - 2 s \rho + s^d B) - \nabla^{\top b}_{\rm e} n^a \IIo^{\rm e}_{ab} + (\nabla^{\top b}_{\rm e} n^a) \IIo^{\rm e}_{ab} \right] \\
\eqSig& \nabla_n^{d-2} \left[K_{\rm e} -n \csdot (\IIo^{\rm e})^2 \csdot n - \rho n \csdot \IIo^{\rm e} \csdot n-  \nabla^{\top}_{\rm e} \csdot n \csdot \IIo^{\rm e}  + s \IIo^{\rm e}_{ab} \left(n^a n^c P_c^b - 2 n^a (\nabla^b \rho) - P^{ab} \right) \right]\,.
\end{align*}
From Lemma~\ref{nIIoe}, we have that if $\IIo^{\rm e} = \mathcal{O}(\sigma^{d-2})$, then $n \csdot \IIo^{\rm e} \stackrel{g}= \mathcal{O}(s^{d-1})$ and $\nabla^\top_{\rm e} \csdot n \csdot \IIo^{\rm e} = \mathcal{O}(s^{d-1})$, and thus $B \eqSig 0$.

To complete the proof, we must show that there exists some $(M^d, \gamma, \sigma) \in \APE{d-4}$ but $(M^d, \gamma,\sigma) \not \in \mathcal{W}_d$. In particular, we consider hypersurface embedding with $\IIo^{\rm e} = \mathcal{O}(\sigma^{d-3})$ but $\IIo^{\rm e} \neq \mathcal{O}(\sigma^{d-2})$.  Observe from Equation~\nn{dn-delt-IIoe} and Remark~\ref{dt-comm} that
$$\nabla_n^{d-2} n \csdot \IIo^{\rm e} \eqSig (d-2) \nabla^\top_{\rm e} \csdot \nabla_n^{d-3} \IIo^{\rm e} + \ltots{d-4}(\IIo^{\rm e})\,.$$
Hence because $\IIo^{\rm e} = \mathcal{O}(\sigma^{d-3})$, we have that
$$\nabla_n^{d-2} \nabla_{\rm e}^\top \csdot n \csdot \IIo^{\rm e} \eqSig (d-2) \nabla^\top \csdot \nabla^\top \csdot \nabla_n^{d-3} \IIo^{\rm e}\,,$$
and hence
$$\frac{d!}{2} B \eqSig - (d-2) \left(\nabla^{\top a} \nabla^{\top b} + P^{ab} \right) \nabla_n^{d-3} \IIo^{\rm e}_{ab} \,.$$

Using the Fialkow--Gau\ss\ equation of~\cite{Will1} and noting that the Fialkow tensor vanishes when $d \geq 5$ because $\IIo^{\rm e} = \mathcal{O}(\sigma^{d-3})$, it then follows that
$$\frac{d!}{2} B \eqSig - (d-2) \left(\bar{\nabla}^a \bar{\nabla}^b + \bar{P}^{ab} \right)(\otop \nabla_n^{d-3} \IIo^{\rm e}_{ab} )\,.$$
We can construct an explicit example to show that such a hypersurface embedding with $\IIo^{\rm e} = \mathcal{O}(\sigma^{d-3})$ but has $\mathcal{B}_d = 0$. To do so, let $U \subset M$ be a neighborhood of $\Sigma$ and work in a choice of scale $\gamma = [g;s]$ such that in a set of local coordinates $(s,y, x^2, \ldots, x^{d-1})$, we have that on $U$ the metric representative takes the form
$$g = ds^2 + 2 s^{d-3} f(y) ds \odot dy + dy^2 + dx^i dx_i\,,$$
for $i=1,\ldots,d-2$.
In this choice of scale, the hypersurface is flat and a tedious but straightforward calculation shows that $\IIo^{\rm e} = \mathcal{O}(s^{d-3})$ but $\IIo^{\rm e} \neq \mathcal{O}(s^{d-2})$, and further that $\mathcal{B}_d \propto  \partial_y^3 f(y) \neq 0$.

In the $d = 4$ case, we can also check using the same example. Using the same embedding described above, we have that $\IIo^{\rm e} = \mathcal{O}(s)$ but $\IIo^{\rm e} \neq \mathcal{O}(s^2)$. Furthermore, using the result for $\mathcal{B}_4$ in~\cite{hypersurface_old}, we can directly compute the Willmore invariant and find that
$$\mathcal{B}_4 = \tfrac{1}{18} \left[(\partial y f)^2 + \partial_y^3 f \right]\,.$$
And hence is nonzero for some function $f(y)$. This completes the proof.
\end{proof}


\subsection{Fundamental forms in even dimensions} \label{new-FFs-sec}
The proof of Theorem~\ref{ape-B} suggests that not only can the Willmore invariant be written in terms of some set of normal derivatives of $\IIo^{\rm e}$ but it specifies exactly how this can be done. Given the conditionality of the canonical $m$th fundamental form for $\frac{d+3}{2} \leq m \leq d-1$ described in the previous section, it is not clear that we can always characterize the Willmore invariant in terms of the fundamental forms. However, results on normal derivative operators from~\cite{GPt} suggest that when $d$ is even, for every $2 \leq m \leq d$, a normal derivative operator exists to produce an $m$th fundamental form unconditionally. However, an exceptional case arises in the construction of a $d$th fundamental form: even though the required normal derivative operator exists to naively produce such a tensor, there is an exceptional identity that forces the transverse order of the candidate tensor to be lower than expected. Note that their results only provide normal derivative operators acting on scalars, so we generalize their work in Theorem~\ref{delta-Jk} to construct normal derivative operators on general tractor tensors. 

First, we record a result directly from~\cite[Proposition 5.8]{GPt}:
\begin{lemma} \label{delta-k}
Let $k \in \mathbb{Z}_+$ be given. Then, there exists a family of conformally invariant operators $\delta_k : \Gamma(\ct^\Phi[w]) \rightarrow \Gamma(\ct^\Phi[w-k])|_{\Sigma}$ defined by
$$\delta_k := N^{A_2} \cdots N^{A_k} \delta_R D_{A_2} \cdots D_{A_k}\,,$$
with transverse order $k$ so long as
$$w \not \in \left\{\tfrac{2k - 1-d}{2}, \tfrac{2k-2-d}{2}, \dots, \tfrac{k+1-d}{2} \right\}\,.$$
Here $\delta_R : \Gamma(\ct^\Phi M[w]) \rightarrow \Gamma(\ct^\Phi M[w-1])|_{\Sigma}$ and is defined by $\delta_R := \nabla_{\hat n} - wH$.
\end{lemma}

\noindent
We now provide a generalization of~\cite[Theorem 5.16]{GPt}.
\begin{theorem} \label{delta-Jk}
Let $J,k \in \mathbb{Z}_+$ such that $0<J$, $0<k<d/2$ and let $d$ be even. Then, there exists a family of conformally invariant differential operators $\delta^\Phi_{J,k} : \Gamma(\ct^\Phi  M[w]) \rightarrow \Gamma(\ct^\Phi M [w-k-J])|_{\Sigma}$ determined as follows.
For $k \leq J$,
$$\delta_{J,k}^{\Phi} = N^{A_1} \cdots N^{A_k} \delta_J P_{A_1 \cdots A_k}^\Phi\,.$$
For $k > J$, then $\delta_{J,k}$ is determined by the equation
$$ (d+2w-2k) \delta^\Phi_{J,k} = N^{A_1} \cdots N^{A_k} \delta_J P_{A_1 \cdots A_k}^\Phi\,.$$
When $w = k-d/2$, $\delta_{J,k}^\Phi$ has transverse order $J+k$.
\end{theorem}
\begin{proof}
We follow the  proof of~\cite[Theorem 5.16]{GPt}. That proof relies on several results that require generalization to arbitrary tractor tensor structures; specifically, it directly requires~\cite[Lemma 5.7, Proposition 5.14, and Lemma 5.15]{GPt}. Of these results,~\cite[Lemmas 5.7 and 5.15]{GPt} already apply to arbitrary tractors, whereas~\cite[Proposition 5.14]{GPt} only applies to scalars and directly relies on~\cite[Proposition 5.10]{GPt}, which also only applies to scalars.~\cite[Proposition 5.10]{GPt}, in turn, relies on~\cite[Proposition 4.5]{GOpet}, which is a result on scalars. Thus, the proof this theorem effectively amounts to generalizing~\cite[Proposition 4.5]{GOpet} to arbitrary tensor structures and then following this result through the aforementioned steps. The generalization of~\cite[Proposition 4.5]{GPt} is given in Proposition~\ref{Box2k}, and the generalizations of~\cite[Propositions 5.10, 5.14]{GPt} are given in Propositions~\ref{P-indices} and~\ref{514-general}, respectively. See Section~\ref{proofs} for these results.
\end{proof}

Such an operator $\delta^\Phi_{J,k}$ is defined on any particular tractor bundle $\ct^\Phi M[w]$, so we will drop the superscript $\Phi$ when the tensor structure is clear from context. A simple consequence of this theorem is the following alternative construction of canonical fundamental forms in even dimensions:
\begin{definition} \label{higher-FFs}
Let $d$ be even and $m$ be an integer satisfying $3 \leq m \leq d-1$. Then the canonical fundamental form $\FF{m}$ is defined by
\begin{align*}
\FF{m}_{ab} :=
\begin{cases} 
     \big(\bar{q}^* \circ \bar{r} \circ \otop \circ \delta_{m-2} \big)(P_{AB}) & m \leq \tfrac{d+2}{2}\\[10pt]
     \big(\bar{q}^* \circ \bar{r} \circ \otop \circ \delta_{\scalebox{0.7}{$m-\tfrac{d+2}{2},\tfrac{d-2}{2}$}} \big)(P_{AB}) & m > \tfrac{d+2}{2}\,.
   \end{cases}
\end{align*}
\end{definition}

We need to check that this definition is correct, given that an $m$th fundamental form is defined to be a symmetric trace-free rank-2 tensor with weight $3-m$ and transverse order $m-1$.
\begin{proposition} \label{higher-dn}
Let $d$ be even. Then for each $m \in \mathbb{Z}_{\geq 2}$ with $m \leq d-1$, the $m$th canonical fundamental form is a fundamental form.
\end{proposition}
\begin{proof}
First we must check that $\FF{m}$ is well-defined. Note that the operator $\bar{r}$ is only ill-defined when $w \in \{0,-1,-\frac{d-1}{2}, -1-\frac{d-1}{2}, -2-\frac{d-1}{2}\}$. Because $\frac{d-1}{2}$ is not an integer and the weights of $\delta_k (P_{AB})$ and $\delta_{J,k}(P_{AB})$ are integers (because the weight of $P_{AB}$ is $-1$), it is clear that application of $\bar{r}$ is permitted  in Definition~\ref{higher-FFs}. 

We next check that $\FF{m}$ has weight $3-m$ as required. Observe that, because $\Im \bar{r} \subset \ker X$, for generic $T \in \Gamma(\odot^2_{\circ} \ct \Sigma [w])$, we have that $(\bar{q}^* \circ \bar{r})(T_{AB}) \in \Gamma(\odot^2_\circ T^* \Sigma[w+2])$. So, by construction, $\FF{m}$ has weight $3-m$.

Finally, we must verify that $\FF{m}$ has transverse order $m-1$. Because the weight of $P_{AB}$ is $-1$, for $m \leq \frac{d+2}{2}$, we can easily check (for even $d$) that $\delta_{m-2}$ has transverse order $m-2$ as per Lemma~\ref{delta-k}. Similarly, we can verify from Theorem~\ref{delta-Jk} that the transverse order of $\delta_{\scalebox{0.7}{$m-\frac{d+2}{2},\tfrac{d-2}{2}$}}$ is $m-2$.

Now we need only verify that the remainder of the definition of $\FF{m}$ does not change the transverse order. To check this, suppose that $\IIo^{\rm e} = \mathcal{O}(\sigma^{m-2})$. Then, because $m-2 \leq d-2$, from Lemma~\ref{nIIoe} and Equation~\nn{delt-dn-IIoe}, we have that $\nabla \csdot \IIo^{\rm e} = \mathcal{O}(\sigma^{m-2})$. However, observe heuristically from Equation~\nn{trac-conn} that the action of the tractor connection on a lower slot either moves the slot ``up'' or applies a derivative. Indeed, for $\IIo^{\rm e} = \mathcal{O}(\sigma^{m-2})$, one can verify by direct computation and Lemma~\ref{nIIoe} that the projecting part of $\nabla_n^{m-2} P_{AB}$ is indeed $\nabla_n^{m-2} \IIo^{\rm e}$. Therefore, for arbitrary $\IIo^{\rm e}$, all of the terms correcting $(\bar{q}^* \circ \bar{r} \circ \otop)(\nabla_n^{m-2} P_{AB})$ must be of lower transverse order than $\nabla_n^{m-2} \IIo^{\rm e}$, \textit{i.e.}
$$(\bar{q}^* \circ \bar{r} \circ \otop) (\nabla_n^{m-2} P_{AB})  = \alpha \otop (\nabla_n^{m-2} \IIo^{\rm e}_{ab}) + \ltots{m-3}(\IIo^{\rm e})\,,$$
for some non-zero $\alpha$.
From~\cite[Proof of Corollary 4.9]{BGW}, we have that 
$$\otop(\nabla_n^{m-2} \IIo^{\rm e}) = \tfrac{d-m}{2(d-2)} \otop (\partial_s^{m-1} g_{ab})\,,$$
and hence has transverse order $m-1$ for $m\leq d-1$. This completes the proof.
\end{proof}

\begin{remark}
This theorem proves the implicit conjecture in~\cite{BGW} that, for $d$ even, there exists an $n$th conformal fundamental form for every integer $2 \leq n \leq d-1$. Furthermore, fundamental forms $\FF{m}$ with $m \geq d+1$ can also be defined using the above result (as well as an extension of Lemma~\ref{nIIoe}). In particular, by uniqueness of the fundamental forms, the tracefree hypersurface projection of the Fefferman-Graham obstruction tensor~\cite{FG,FGbook} of $(M^d,\cc)$  is $\FF{d+1}$. As a first example, the Fefferman-Graham obstruction tensor in four dimensions is the Bach tensor whereas according to~\cite{BGW}, in four dimensions we have that $\Vo = B_{(ab)\circ}^\top$.
\end{remark}

Because such fundamental forms exist, going forward we will represent by $\FF{m}$ a normalized conformal fundamental form such that it has identical structure to the $m$th canonical fundamental form defined above but its leading derivative structure has a coefficient of $1$, so that $\FF{m} = \nabla_n^{m-2} \IIo^{\rm e} + \ltots{m-3}(\IIo^{\rm e})$.

Also observe that even though the operator $\delta_{\scalebox{0.7}{$\tfrac{d-2}{2},\tfrac{d-2}{2}$}}$ has transverse order $d-2$, the calculation in the above proof shows that
\begin{align} \label{dth-FF}
\nabla_n^{d-2} \IIo^{\rm e}_{ab} \eqSig \ltots{d-2}(g_{ab})\,,
\end{align}
which shows that the implied construction of the $d$th fundamental form above is not tenable. An alternative construction, using a conformally-invariant extension of $\IIIo$ given by
$$\IIIo^{\rm e}_{ab} := W_{nabn} + 2 \sigma C_{n(ab)} - \tfrac{\sigma^2}{d-4} B_{ab}\,,$$
yields the same phenomenon when attemping to construct the $d$th fundamental form for $d \geq 6$ even. Indeed, in a choice of scale where $|ds|^2 = 1$ and a choice of coordinates $(s,y^1, \ldots, y^{d-1})$, we have that
$$\otop \circ \nabla_n^{m} \IIIo^{\rm e}_{ab} = \tfrac{(d-m-3)(d-m-4)}{2(d-2)(d-4)} \partial_s^{m+2} g_{ab} + \ltots{m+1}(g_{ab})\,,$$
so that for $d \geq 6$, 
$$\nabla^{d-3}_n \IIIo^{\rm e}_{ab} = \ltots{d-2}(g_{ab})\,.$$
However, in $d =4$, the fourth fundamental form $\IVo$ is constructible because the leading derivative has a removable singularity at $d=4$, and thus we can define
$$\IVo := (\bar{q}^* \circ \bar{r} \circ \otop \circ \delta_R \circ q)(\IIIo^{\rm e})\,.$$

\begin{remark}
The above alternative construction of fundamental forms (using an extension of $\IIIo$) suggests that the higher conformal fundamental forms only exist on conformal manifolds that are not conformally flat.
\end{remark}

For generally curved conformal manifolds $(M,\cc)$ in even dimensions, the cancellation of the leading derivative term on $g_{ab}$ suggests that there exists no conformal fundamental form $\FF{d}$ so long as $\IIo \neq 0, \ldots, \FF{d-1} \neq 0$. At present we have no evidence for or against the existence of $d$th fundamental forms and leave this as an open question.

\medskip

A consequence of the proof of Theorem~\ref{ape-B} follows from Definition~\ref{higher-FFs} in even dimensions.
\begin{corollary} \label{B-makeup}
Let $(M^d, \gamma,\sigma)_{\mathcal{Y}}$ specify a conformally embedded hypersurface with $d \geq 4$ even. Then, the Willmore invariant $\mathcal{B}_d$ has transverse order $d-1$ and can be expressed in such a way that each term contains at least one element of the set $\{\IIo, \ldots,\FF{d-1}, \otop \nabla_n^{d-3} P\}$ with a non-zero coefficient and $\mathcal{B}_d$ can be written to explicitly depend polynomially on each element of the set.
\end{corollary}
\begin{proof}

We first check that Corollary~\ref{B-makeup} holds for $d = 4$.  From an explicit formula in~\cite{hypersurface_old}, we see that $\mathcal{B}_4$ contains exactly one term that contains $C_{\hat{n} ab}^\top$, which has transverse order $3$. Further we note that $\mathcal{B}_4$ can be written in a way such that each term contains at least one of $\IIo, \IIIo, C_{\hat{n} ab}^\top$. However note that $C_{\hat{n} ab}^\top = \otop \nabla_n P_{ab} + \ltots{0}(P_{ab})$, so the corollary holds.

Now suppose that $d \geq 6$. We check that the transverse order of $\mathcal{B}_d$ is $d-1$ first. From the proof of Theorem~\ref{ape-B}, the terms with leading normal derivatives on $\IIo^{\rm e}$ are $\nabla_n^{d-2} (K_{\rm e} - s \IIo^{\rm e} \csdot P)$. Observe that $\nabla_n^{d-2} K = \IIo \csdot \nabla_n^{d-2} \IIo^{\rm e} + \ltots{d-3}(\IIo^{\rm e})$. However, from Equation~\nn{dth-FF} we see that the leading transverse order term of this expression vanishes. So we need only consider $\nabla_n^{d-2} (s \IIo^{\rm e} \csdot P)$. A direct calculation shows that this expression has transverse order $d-1$. But because this is the only term in $\mathcal{B}_d$ with that transverse order, we have that $\mathcal{B}_d$ has transverse order $d-1$. 

From the proof of Theorem~\ref{ape-B}, it is clear that each term can be expressed in terms of some operator acting on $\IIo^{\rm e}$ and hence from Proposition~\ref{higher-dn} and Lemma~\ref{nIIoe} can be expressed in terms of a canonical fundamental form---with the exception of terms with transverse order $d-1$. In particular, for any integer $\ell$ satisfying $0 \leq \ell \leq d-3$, we have that
\begin{equation} \label{leading-dn}
\nabla_n^{\ell} \IIo^{\rm e} \eqSig \FF{\ell+2} + \ltots{\ell-1}(\IIo^{\rm e})\,.
\end{equation}
As noted above, the term with transverse order $d-1$ only arises in the product $\IIo \csdot \nabla_n^{d-3} P$, and hence we can characterize that term by $\otop \nabla_n^{d-3} P$. Thus, we have that each summand in $\mathcal{B}_d$ contains at least one element of the set $\{\IIo, \ldots, \FF{d-1}, \otop \nabla_n^{d-3} P\}$.

\medskip

It now remains to show that $\mathcal{B}_d$ can be written in such a way that it explicitly depends on each element of the set $\{\IIo, \ldots, \FF{d-1}, \otop \nabla_n^{d-3} P\}$ with a non-zero coefficient.
A key observation is that only the terms in $\mathcal{B}_d$ of the form $\nabla_n^{d-2} (K_{\rm e} - s \IIo^{\rm e} \csdot P)$ can contain terms that are a product of exactly two elements from this set. One way to see this is that $n \csdot \nabla_n \IIo^{\rm e}$ can be written in terms of $\nabla^\top_{\rm e} \csdot \IIo^{\rm e}$, and that derivative cannot be eliminated by any manipulations. So we now consider an expansion of this difference:
\begin{align} \label{FF-squared}
\begin{split}
\nabla_n^{d-2} (K_{\rm e} - s \IIo^{\rm e} \csdot P) \eqSig \sum\limits_{k= 0}^{d-2} \tbinom{d-2}{k} (\nabla_n^k \IIo^{\rm e}) &\csdot (\nabla_n^{d-k-2} \IIo^{\rm e})  \\&-(d-2) \sum\limits_{k = 0}^{d-3} \tbinom{d-3}{k} (\nabla_n^{k} \IIo^{\rm e}) \csdot (\nabla_n^{d-k-3} P) + \text{more},
\end{split}
\end{align}
where $\text{more}$ contains no terms quadratic in elements of the set $\{\nabla_n^m \IIo^{\rm e}, \nabla_n^{\ell} P\}_{m, \ell}$ for $0\leq m \leq d-2$ and $0\leq \ell \leq d-3$. Before we proceed, note that the term of the form $P \csdot \nabla_n^{d-3} \IIo^{\rm e}$ in the display above can be paired with $\nabla^\top \csdot \nabla^\top \csdot \nabla_n^{d-3} \IIo^{\rm e}$ (resulting from manipulating the term $\nabla_n^{d-2} \nabla^\top_{\rm e} \csdot n \csdot \IIo^{\rm e}$) to form a conformal invariant. In particular, observe that the operator $L$ defined in~\cite{hypersurface_old} by
$$\Gamma(\odot^2_\circ T^* \Sigma[-d+4]) \ni T_{ab} \mapsto (\bar \nabla^a \bar \nabla^b + \bar{P}^{ab}) T_{ab} \in \Gamma(\ce \Sigma[-d])$$
is conformally invariant on weight $4-d$ densities, which is precisely the weight of $\FF{d-1}$. But the leading term of $\nabla_n^{d-3} \IIo^{\rm e}$ is indeed $\FF{d-1}$ as per Equation~\nn{leading-dn}, so the leading term of
$$(\nabla^{\top a}\nabla^{\top b} + P^{ab}) \nabla_n^{d-3} \IIo^{\rm e}_{ab}$$
 is precisely $L(\FF{d-1})$. Thus, $\mathcal{B}_d$ contains a term of the form $L(\FF{d-1})$ which absorbs the term of the form $P \csdot \nabla_n^{d-3} \IIo^{\rm e}$. Thus, from Equation~\nn{FF-squared}, we are only interested in the terms
\begin{align} \label{FF-squared2}
\sum\limits_{k= 0}^{d-2} \tbinom{d-2}{k} (\nabla_n^k \IIo^{\rm e}) &\csdot (\nabla_n^{d-k-2} \IIo^{\rm e})  -(d-2) \sum\limits_{k = 0}^{d-4} \tbinom{d-3}{k} (\nabla_n^{k} \IIo^{\rm e}) \csdot (\nabla_n^{d-k-3} P).
\end{align}

We now work in the scale where $|ds|_g = 1$ and choose a set of coordinates $(s,y^1, \ldots, y^{d-1})$, so that
\begin{align*}
\nabla_n^{m} \IIo^{\rm e} =& \; \tfrac{d-m-2}{2(d-2)} \partial_s^{m+1} g_{ab} + \ltots{m-2}(g)_{ab} \\
 \nabla_n^{m}P_{ab} =& \; -\tfrac{1}{2(d-2)} \partial_s^{m+2} g_{ab} + \ltots{m+1}(g)_{ab}\,.
\end{align*}
Applying the above display to Display~\nn{FF-squared2} and keeping only terms quadratic in $\partial_s^m g$, we have
\begin{align*}
\tfrac{1}{4} (\partial_s g) \csdot ( \partial_s^{d-1} g) + \tfrac{(d-1)(d-3)}{4(d-2)} (\partial_s^2 g) \csdot ( \partial_s^{d-2} g) + \tfrac{1}{4} \sum_{k=2}^{d-4} \tbinom{d-3}{k} ( \partial_s^{k+1} g) \circ (\partial_s^{d-k-1} g)\,.
\end{align*}
The definition of a fundamental form is a conformally-invariant rank-$2$ tensor of the appropriate weight and transverse order. Because $\mathcal{B}$ is conformally invariant, a consequence of Equation~\nn{leading-dn} is that we must be able to express the above display in terms of quadratic products of fundamental forms of the form $\FF{k+2} \csdot \FF{d-k}$ plus subleading terms---with the exception of the term with transverse order $d-1$. But because none of the coefficients vanish for $d \geq 4$ in the above display, the corresponding coefficients for $\FF{k+2} \csdot \FF{d-k}$ must also not vanish for each product for $1 \leq k \leq d-3$. For the same reason as the terms quadratic in fundamental forms, the one remaining term, of the form $\tfrac{1}{4} (\partial_s g) \csdot (\partial_s^{d-1} g)$ can be written in the form $\IIo \csdot \nabla_n^{d-3} P$ plus subleading terms with a non-zero coefficient.
%
%
%
%
Thus, $\mathcal{B}_d$ can be written so that it explicitly depends on every element of the set  $\{\IIo, \ldots,\FF{d-1}, \otop \nabla_n^{d-3} P\}$ with a non-zero coefficient. This completes the proof.
\end{proof}

\begin{remark}
Much of the argument in the proof of Corollary~\ref{B-makeup} follows when $d \geq 5$ is odd. In that case, for $n \geq \frac{d+3}{2}$, the fundamental forms are only conformally-invariant when $\FF{m} = 0$ for each $2 \leq m < \frac{d+3}{2}$. Nonetheless, the higher transverse derivative terms $\nabla_n^{k} \IIo^{\rm e}$ still appear in $\mathcal{B}_d$, so a similar statement as the above corollary would hold in terms of what may be called \textit{pre-fundamental forms}---tensors that are not conformally-invariant but become conformally-invariant when $\FF{m} = 0$ for each $m$ in the range $2 \leq m < \frac{d+3}{2}$.
%
\end{remark}

\section{Proof of Theorem~\ref{delta-Jk}} \label{proofs}

Our goal in this section is to prove the necessary results to generalize~\cite[Theorem 5.16]{GPt} and hence prove Theorem~\ref{delta-Jk}. Most of the required proofs go through without significant changes to the original scalar proofs, so we summarize the work of~\cite{GOpet,GPt} and provide details only when necessary.

To prove Theorem~\ref{delta-Jk}, we rely on a construction known as the \textit{ambient construction}, first introduced by Fefferman and Graham in~\cite{FG,FGbook}. We now briefly summarize the key ingredients of the ambient construction. For more detail, see~\cite{GW}.

Given a conformal manifold $(M^d,\cc)$, the conformal class of metrics can be viewed as a ray subbundle $\pi : \mathcal{Q} \rightarrow M$ of metrics on $M$. The natural $\mathbb{R}_+$ action, denoted by $\rho$ and parametrized by $t$, on an element of this ray subbundle then is $\rho_t(x,g_x) = (x,t^2 g_x)$. This subbundle carries with it a tautological symmetric $2$-tensor $g_0 := \pi^*g$ at some point $(x,g_x) \in \mathcal{Q}$ which obeys $\rho^*_t (g_0) = t^2 g_0$. The associated \textit{ambient manifold} then is a $d+2$-dimensional manifold $\aM$ (with signature $(d+1,1)$) given by $\mathcal{Q} \times (-1,1)$ with the $\mathbb{R}_+$ action on $\mathcal{Q}$ naturally extended to $\aM$. We denote by $\X$ the infinitesimal generator of the $\mathbb{R}_+$ action and we will denote by $Q$ the defining function for $\mathcal{Q}$.

The metric $\h$ on $\aM$ is defined such that $\h$ pulls back to $g_0$ on  $\mathcal{Q}$ and has the same homogeneity property, \textit{i.e.} $\rho^*_s \h= s^2 \h$. Further, we demand that  the corresponding Ricci curvature $\aRic$ vanishes to the maximal possible order. In particular, because we are only interested in the even dimensional case, a result of~\cite{FG} implies that we can at most uniquely determine $\h$ to order $\mathcal{O}(Q^{d/2})$. Note that this fixes $\aR$ (the ambient Riemann curvature) and $\aRic$ uniquely to order $\mathcal{O}(Q^{d/2-2})$.

Given this construction, there exists a injective relationship between certain operators and tensors on in $\aM|_{\mathcal Q}$ and tractors. It is this relationship that underlies the proofs given in~\cite{GOpet} which are required here. For a homogeneous weight $w$ tensor field $T$ of type $\Phi$ on $\aM|_{\mathcal{Q}}$, we will denote the section space by $\Gamma(\ct^\Phi \aM(w))|_{\mathcal Q}$, using similar notation as that used for tractors due to the injective relationship just described.  In particular, the ambient operator given by
$$\aD_A := \boldnabla_A (d+2 \boldnabla_{\X} - 2) + \X_A \al$$ maps to the Thomas-$D$ operator and
$$\al|_{\mathcal{Q}} : \Gamma(\ct^\Phi \aM(1-d/2))|_{\mathcal Q} \rightarrow \Gamma(\ct^\Phi \aM(-1-d/2))|_{\mathcal Q}$$
 maps to the \textit{Yamabe operator} $\square$, defined according to $D_A = -X_A \square$ when acting on tractors with weight $1-d/2$. The fundamental vector field $\X$ maps to $X$, $\h$ maps to the tractor metric $h$, and $\aR$ maps to the $W$-tractor.

A key result is the straight forward generalization of~\cite[Proposition 4.1]{GOpet} (which in turn is a rewriting of~\cite[Proposition 2.3]{GJMS}). Note that the result directly generalizes with no changes.
\begin{proposition}[Generalization of Proposition 4.1 of~\cite{GOpet}] \label{tractor-GJMS}
For $n$ even and $k$ an integer satisfying $0 < k < n/2$, let $T \in \Gamma(\ct^\Phi \mathcal{Q}(k-d/2))$ with a homogeneous extension $\tilde{T}$ to $\aM$. Then $\al^k \tilde{T}|_{\mathcal{Q}}$ depends only on $T$ and the conformal structure on $M$ but not on the choice of extension $\tilde{T}$ nor on any choices in the ambient metric. Thus the operator
$$\al^k : \Gamma(\ct^\Phi \aM (k-d/2))|_{\mathcal Q} \rightarrow \Gamma(\ct^\Phi \aM(-k-d/2))|_{\mathcal Q}$$
is conformally invariant and descends to a natural conformally invariant differential operator on tractors
$$P^\Phi_{2k} : \Gamma(\ct^\Phi M[k-d/2]) \rightarrow \Gamma(\ct^\Phi M[-k-d/2])\,.$$
\end{proposition}
The operators $P_{2k}$ are known as GJMS operators, named after Graham, Jenne, Mason, and Sparling, who demonstrated their existence. This result completes the injection that we require to relate ambient calculus to the tractor calculus.

We need one more technical lemma from~\cite{GOpet}, so it is recorded here for later use:
\begin{lemma}[Lemma 4.4 of \cite{GOpet}] \label{lemma44}
Suppose $t + u \leq d/2 - 3$ for $d$ even. Then on $\mathcal{Q}$ there is an expression for $\boldnabla^t \al^u \aR$ as a partial contraction polynomial in $\aD$, $\aR$, $\X$, $\h$, and its inverse. This expression is rational in $d$ and each term is of degree at least $1$ in $\aR$.
\end{lemma}

With Proposition~\ref{tractor-GJMS} and Lemma~\ref{lemma44}, we can characterize a tractor formula for the GJMS operators using the ambient construction. The proof of this result follows the proof of~\cite[Theorem 4.5]{GOpet} closely.
\begin{proposition}[Generalization of Proposition 4.5 of~\cite{GOpet}]\label{Box2k}
Let $k < d/2$ with $d$ even. For any tractor $T \in \Gamma(\ct^\Phi M[k-d/2])$,
$$X_{A_1} \cdots X_{A_{k-1}} P^\Phi_{2k} T= (-1)^{k-1} \square D_{A_1} \cdots D_{A_{k-1}} T + \Psi_{A_1 \cdots A_{k-1}}{}^C D_C T\,,$$
where $\Psi$ is some tractor opperator operator
$$\Psi_{A_1 \cdots A_{k-1}}{}^C : \Gamma(\ct_C^\Phi M[k-1-d/2]) \rightarrow \Gamma(\ct_{A_1 \cdots A_{k-1}}^\Phi M[-1-d/2])\,,$$
that can be written as a contraction polynomial in $D_A$, $W_{ABCD}$, $X_A$, $h_{AB}$, and $h^{AB}$, with each term being of order at least $1$ in $W_{ABCD}$.
\end{proposition}

\begin{proof}

The proof of this result follows the discussion (beginning on page 36) in~\cite{GOpet}. The strategy will be to find a relationship between  ambient operators and then descend to their corresponding tractor counterparts.

Given the ambient construction described above for an even dimensional base manifold $M^d$, we fix an integer  $0<k < d/2$ and a homogeneous ambient tensor $T$, fixed along $\mathcal{Q}$, of weight $k-d/2$ with an arbitrary extension, also labeled by $T$. Then, we consider the ambient expression $\al \aD_{A_1} \cdots \aD_{A_{k-1}} T$. Our goal will be to rearrange terms so that we have
$$\al \aD_{A_1} \cdots \aD_{A_{k-1}} T = (-1)^{k-1} \X_{A_1} \cdots X_{A_{k-1}} \al^k T + \text{curvature terms}.$$
Observe that $\boldnabla_{\X}$ acting on a homogeneous tensor $T$ of weight $w$ is $wT$, so we can write
\begin{equation} \label{starting-point}
\al \aD_{A_1} \cdots \aD_{A_{k-1}} T = \al (2 \boldnabla_{A_{k-1}} - X_{A_{k-1}} \al) \cdots (2(k-1) \boldnabla_{A_1} - X_{A_1} \al) T\,.
\end{equation}
Our first step is to expand this display and then commute all of the $\X$s to the left of any $\boldnabla$s and $\al$s using the identities $[\boldnabla_A, \X_B] = \h_{AB}$ and $[\al, \X_A] = 2 \boldnabla_A$. Our next step will be to commute all $\al$s to the right of any $\boldnabla$s; doing so requires the following operator identity:
\begin{align}\label{commLapNabla}
\begin{split}
\al \boldnabla_{A_1} \cdots \boldnabla_{A_{\ell}} =&
 \boldnabla_{A_1} \cdots \boldnabla_{A_{\ell}} \al \\
&+ 2 \aR^C{}_{A_1}{}^\sharp \boldnabla_C \boldnabla_{A_2} \cdots \boldnabla_{A_{\ell}} \\
&+ 2 \boldnabla_{A_1} \aR^C{}_{A_2}{}^\sharp \boldnabla_C \boldnabla_{A_3} \cdots \boldnabla_{A_{\ell}} \\
&+ \cdots \\
&+ 2 \boldnabla_{A_1} \cdots \boldnabla_{A_{\ell -1}} \aR^C{}_{A_{\ell}}{}^\sharp \boldnabla_C + \mathcal{O}(Q^{d/2 - \ell - 1})\,,
\end{split}
\end{align}
where $\aR_{AB}{}^\hash V^C = \aR_{AB}{}^C{}_D V^D$ and $\hash$ extends naturally as a derivation. Note that this identity holds to order $\mathcal{O}(Q^{d/2-\ell-1})$ because $\aR$ is only determined uniquely to order $\mathcal{O}(Q^{d/2-2})$. Going forward, we will use the simplifying notation of~\cite{GOpet}, which loses the details of contractions and coefficients in exchange for symbolic brevity. For example, applying the Leibniz rule, we can write the result in Equation~\nn{commLapNabla} as
$$\al \boldnabla^{\ell}  = \boldnabla^{\ell} \al  + \sum (\boldnabla^p \aR) \boldnabla^q \,,$$
where each term on the right-hand side of the above equation has $q \geq 1$ and $p + q = \ell$. Observe from simple counting that any term of the form $\al \boldnabla^\ell$ appearing in Equation~\nn{starting-point} has $\ell < k$, so because $k < d/2$, we have that using Equation~\nn{commLapNabla} is valid in all cases here. In order to further simplify our results, we need another identity, which applies for any expression $E$:
\begin{equation} \label{al-through}
\al (\boldnabla^t \al^u \aR) E = (\al \boldnabla^t \al^u \aR) E + (\boldnabla^t \al^u \aR) \al E + (\boldnabla^{t+1} \al^u \aR) \boldnabla E\,.
\end{equation}

Recall from~\cite[Proposition 4.3]{GOpet} that for a conformally flat structure, if $T \in \Gamma(\ct^\Phi \aM(k-d/2))$, then
$$\al \aD_{A_{k-1}} \cdots \aD_{A_1} T = (-1)^{k-1} \X_{A_1} \cdots \X_{A_{k-1}} \al^k T\,.$$
Therefore, except the term $(-1)^{k-1} \X_{A_1} \cdots \X_{A_{k-1}} \al^k$, all of the terms remaining after commuting $\X$s to the left except must contain at least one curvature $\aR$ (because $\aR$ vanishes for conformally flat structures).
Thus, we can write 
\begin{equation} \label{point2}
\resizebox{0.9\linewidth}{!}{$
\al \aD_{A_1} \cdots \aD_{A_{k-1}} T = (-1)^{k-1} \X^{k-1} \al^k T + \sum \h^s \X^x (\boldnabla^{p_1} \al^{r_1} \aR) \cdots (\boldnabla^{p_n} \al^{r_n} \aR) \boldnabla^q \al^r T \,,
$}
\end{equation}
where $n \geq 1$ for each term on the right-hand side.

We now apply some counting arguments to check that our expression is uniquely determined. With the exception of the application of Equation~\nn{al-through}, all of the identities we have applied cannot increase the sum of the number of $\boldnabla$s, $\al$s, and $\aR$s, represented by $n+q+r$. On the other hand, Equation~\nn{al-through} increases $q+r$ by at most one at the cost of one $\al$ acting from the left. However, we had at most $k$ such symbols initially, so we observe that $n+q+r \leq k$ and hence $k - q - r \geq 1$. Put another way, $q + r < d/2 -1 $. Because $\boldnabla$ is determined modulo terms of order $\mathcal{O}(Q^{d/2-1})$ and $\al$ is determined modulo terms of order $\mathcal{O}(Q^{d/2-1})$, and because $[\al, Q] = 2(d + 2 \boldnabla_{\X} + 2)$, we  have that the operator $\boldnabla^q \al^r$ is uniquely determined modulo terms of order $\mathcal{O}(d/2-q-r)$. Thus, $\boldnabla^q \al^r$ as an operator on $T$ is uniquely determined along $\mathcal{Q}$.

Further, observe that because there is always at least one $\boldnabla$ on the right in Equation~\nn{commLapNabla}, we have that $q \geq 1$, and so at most $q-1$ $\boldnabla$s on $T$ can come from Equation~\nn{al-through} or from the application of the typical Leibniz rule. Thus, $p_i + r_i \leq q-1$, because the Leibniz rule and Equation~\nn{al-through} are the only way to get more derivatives on $\aR$. But because $k-q-r \geq 1$ and $r \geq 0$, we have that $p_i + r_i + 2 \leq k$ and hence $p_i + r_i < d/2-2$. Thus, by a similar argument as above, $\boldnabla^{p_i} \al^{r_i} \aR$ is determined uniquely up to order $\mathcal{O}(Q^{d/2-2})$ and so we have that all of the terms containing curvature are uniquely determined. 

A straightforward extension of~\cite[Proposition 2.2]{GJMS} to general tensor structures implies that, given some homogeneous tensor $T$ of weight $k-d/2$ along $Q$, $T|_{\mathcal Q}$ uniquely (to order $\mathcal{O}(Q^k)$) determines a canonical extension $\tilde{T}$ by requiring that $\al \tilde{T} = \mathcal{O}(Q^{k-1})$. Thus, to make our calculations easier, we choose $T := \tilde{T}$ that is harmonic in this sense. With this prescription in place, we have that $\boldnabla^q \al^r \tilde{T} = \mathcal{O}(Q^{k-r-q})$ and thus vanishes along $Q$. So we will drop all terms in Equation~\nn{point2} with $r > 0$.

A useful identity that holds regardless of tensor type is~\cite[Equation 42]{GOpet}, recorded here for our use. Acting on $\tilde{T} \in \Gamma(\ct^\Phi \aM(k-d/2))$, we have that
\begin{equation} \label{eqn42}
2(k-\ell-1) \boldnabla^{\ell +1} \tilde{T} = \aD \boldnabla^\ell \tilde{T} + \X \sum (\boldnabla^p \aR) \boldnabla^q \tilde{T} + \X \boldnabla^\ell \al \tilde{T}\,,
\end{equation}
where $q \geq 1$ and $p+q = \ell$. Note that this identity follows from Equation~\nn{commLapNabla}. But because $\al \tilde{T} = \mathcal{O}(Q^{k-1})$, we can drop the last term on the right-hand side when using this identity. Because $p_i + r_i + 2 \leq k$ and $k < d/2$, we have that $p_i + r_i \leq d/2 -3$, so we can safely apply Lemma~\ref{lemma44} and Equation~\nn{eqn42} to substitute all occurrences of $\boldnabla$ and $\al$ in Equation~\nn{point2} with $\aD$, $\X$, and $\aR$. Therefore, we can write
\begin{align*}
(-1)^{k-1} \X^{k-1} \al^k \tilde{T} = \al \aD_{A_1} \cdots \aD_{A_{k-1}} \tilde{T} + \Psi (\tilde{T}) \,, 
\end{align*}
where $\Psi$ is an operator that is polynomial in $\h, \X, \aD, \R$ and rational in $d$. Note that $\Psi$ must end in $\aD$ because $q \geq 1$ everywhere. Then note that each of these operators on the right-hand side descend to their tractor counterparts. From Proposition~\ref{tractor-GJMS}, the left-hand side descends to the desired tractor operator $P_{2k}^{\Phi}$ which completes the proof.

For a more detailed exposition, see~\cite[Section 4]{GOpet}.

\end{proof}

\noindent
We now rewrite the operator $\Psi$ in Proposition~\ref{Box2k} in the following proposition:
\begin{proposition}[Generalization of Proposition 5.10 of~\cite{GPt}] \label{P-indices}
Let $k < d/2$ and $d$ even. For any tractor $T \in \Gamma(\ct^\Phi M[k-d/2])$, there exists a conformally invariant operator
$$\mathcal{P}^{\Phi,k}_{A_1 \cdots A_{k-1}} : \Gamma(\ct^{\Phi} M[k-d/2]) \rightarrow \Gamma(\ct_{A_1 \cdots A_{k-1}}^\Phi M[-1-d/2])$$
such that
$$ (-1)^{k-1}  X_{A_1} \cdots X_{A_{k-1}} P_{2k}^\Phi T=\square D_{A_1} \cdots D_{A_{k-1}} T + \mathcal{P}^{\Phi,k}_{A_1 \cdots A_{k-1}} T\,,$$
where $\mathcal{P}^{\Phi,k}_{A_1 \cdots A_{k-1}}$ has a tractor formula, is at least order $1$ in $W$,
and has weight $-k-1$.
\end{proposition}
\begin{proof}
By identifying $\mathcal{P}^{\Phi,k}_{A_1 \cdots A_{k-1}}$ with $(-1)^{k-1} \Psi_{A_1 \cdots A_{k-1}}{}^C D_C$, the proposition follows.
\end{proof}

\medskip

\noindent
Proposition~\ref{P-indices} directly implies the following proposition:
\begin{proposition}[Generalization of Proposition 5.14 of~\cite{GPt}] \label{514-general}
There exists a family of operators
$$P^{\Phi}_{A_1 \cdots A_{k}} : \Gamma(\ct^\Phi M[w]) \rightarrow \Gamma(\ct_{A_1 \cdots A_{k}}^\Phi M[w-k])$$
defined by
$$ P^{\Phi}_{A_1 \cdots A_{k}} T = D_{A_1} \cdots D_{A_k} T - X_{A_1} \mathcal{P}^{\Phi,k}_{A_2 \cdots A_{k}} T$$
for all $T \in \Gamma(\ct^\Phi M[w])$. Then, we have that when $w = k-d/2$,
$$ P^{\Phi}_{A_1 \cdots A_{k}} T = (-1)^k X_{A_1} \cdots X_{A_k} P^\Phi_{2k} T\,.$$
\end{proposition}
\begin{proof}
Observe that there exists a tractor formula for $\mathcal{P}^{\Phi,k}_{A_1 \cdots A_{k-1}}$ in terms of $h$, $X$, $D$, and $W$, so combining that tractor formula with $\hd_{A_1} \cdots \hd_{A_{k}}$ allows us to construct a well-defined operator $P^\Phi_{A_1 \cdots A_k}$ as in the Proposition statement. Then, observe that acting on tractors of weight $1-d/2$, $D_A = -X_A \square$, so the remainder of the proposition follows from Proposition~\ref{Box2k} and Proposition~\ref{P-indices}.
\end{proof}

\section*{Acknowledgements}

S.~Blitz would like to acknowledge fruitful discussions with A.~Rod Gover and Andrew Waldron.

\color{black}


\begin{thebibliography}{10}



\bibitem{ACF}
L. Andersson, P.~T. Chru\'{s}ciel, and H. Friedrich.
\newblock On the regularity of solutions to the {Y}amabe equation and the
  existence of smooth hyperboloidal initial data for {E}instein's field
  equations.
\newblock {\it Comm. Math. Phys.}, 149(3):587--612, 1992.

\bibitem{AGW} 
C. Arias, A. R.  Gover, A. Waldron.
Conformal geometry of embedded manifolds with boundary from universal holographic formul\ae.  
{\it Adv. Math.} 384: 107700-107775 (2021),
arXiv:1906.01731.


\bibitem{Aviles}
P. Aviles and R.~C. McOwen.
\newblock Complete conformal metrics with negative scalar curvature in compact
  {R}iemannian manifolds.
\newblock {\it Duke Math. J.}, 56(2):395--398, 1988.


\bibitem{BEG}
T.~N. Bailey, M.~G. Eastwood, and A.~R. Gover.
\newblock T.'s structure bundle for conformal, projective and related
  structures.
\newblock {\it Rocky Mountain J. Math.}, 24(4):1191--1217, 1994.



\bibitem{BGW}
S. Blitz, A.~R. Gover, and A. Waldron.
Conformal fundamental forms and the asymptotically Poincar\'e--Einstein condition.
arXiv:2107.10381.


\bibitem{BGW2}
S. Blitz, A.~R. Gover, and A. Waldron.
Generalized Willmore energies, $Q$-curvatures, extrinsic Paneitz operators, and extrinsic Laplacian powers.
arXiv:2111.00179.


















\bibitem{FG}
C. Fefferman and C.~R. Graham.
\newblock Conformal invariants.
\newblock Number {N}um\'{e}ro {H}ors {S}\'{e}rie, pages 95--116. 1985.
\newblock The mathematical heritage of \'{E}lie Cartan (Lyon, 1984).

\bibitem{FGbook}
C. Fefferman and C.~R. Graham.
\newblock {\it The ambient metric}, volume 178 of {\it Annals of Mathematics
  Studies}.
\newblock Princeton University Press, Princeton, NJ, 2012.




\bibitem{hypersurface_old}
M. Glaros, A.~R. Gover, M. Halbasch, and A. Waldron.
\newblock Variational calculus for hypersurface functionals: singular {Y}amabe
  problem {W}illmore energies.
\newblock {\it J. Geom. Phys.}, 138:168--193, 2019.





\bibitem{Goal}
A.~R. Gover.
\newblock Almost {E}instein and {P}oincar\'{e}-{E}instein manifolds in
  {R}iemannian signature.
\newblock {\it J. Geom. Phys.}, 60(2):182--204, 2010.


\bibitem{GOpet}
A.~R. Gover and L.~J. Peterson.
\newblock Conformally invariant powers of the {L}aplacian, {$Q$}-curvature, and
  tractor calculus.
\newblock {\it Comm. Math. Phys.}, 235(2):339--378, 2003.

\bibitem{GPt}
A.~R. Gover and L.~J. Peterson.
\newblock Conformal boundary operators, $t$-curvatures, and conformal
  fractional {L}aplacians of odd order.
\newblock {\it Pac. J. Math. in press}, 2021.

\bibitem{GOmin}
A.~R. Gover and A. Waldron.
\newblock The {$\frak{so}(d+2,2)$} minimal representation and ambient tractors:
  the conformal geometry of momentum space.
\newblock {\it Adv. Theor. Math. Phys.}, 13(6):1875--1894, 2009.

\bibitem{GW}
A.~R. Gover and A. Waldron.
\newblock Boundary calculus for conformally compact manifolds.
\newblock {\it Indiana Univ. Math. J.}, 63(1):119--163, 2014.

\bibitem{Will1}
A.~R. Gover and A. Waldron.
\newblock Conformal hypersurface geometry via a boundary
  {L}oewner--{N}irenberg--{Y}amabe problem.
\newblock {\it Commun. Anal. Geom. to appear}, 2015.




\bibitem{Will2}
A.~R. Gover and A. Waldron.
\newblock A calculus for conformal hypersurfaces and new higher {W}illmore
  energy functionals.
\newblock {\it Adv. Geom.}, 20(1):29--60, 2020.

 



\bibitem{GJMS}
C. R. Graham, R.  Jenne, L. Mason, 
              G. A. J. Sparling, 
      {Conformally invariant powers of the {L}aplacian. {I}.
              {E}xistence},
   {\it J. London Math. Soc. (2)},
     {46}: 557--565,
     {1992}.
     

\bibitem{GR}
C.R. Graham and N. Reichert.
Higher-dimensional Willmore energies via minimal submanifold asymptotics.
{\it Asian J. Math.} 24: 571--610 (2020), arXiv:1704.03852.

\bibitem{GrWi}
C. R. Graham and E. Witten. Conformal anomaly of submanifold observables in AdS/CFT
correspondence. {\it Nucl. Phys. B} 546: 52--64 1999, arXiv:hep-th/9901021.



\bibitem{Guven}
J. Guven.
Conformally invariant bending energy for hypersurfaces. {\it J. Phys. A} 38: 7943--7956 (2005), arXiv:cond-mat/0507320.








\bibitem{LeBrun}
C.~R. LeBrun.
\newblock {${\mathcal H}$}-space with a cosmological constant.
\newblock {\it Proc. Roy. Soc. London Ser. A}, 380(1778):171--185, 1982.


\bibitem{Loewner}
C. Loewner and L. Nirenberg.
\newblock Partial differential equations invariant under conformal or
  projective transformations.
\newblock In {\it Contributions to analysis (a collection of papers dedicated
  to {L}ipman {B}ers)}, pages 245--272. 1974.


\bibitem{Maz}
R. Mazzeo.
\newblock Regularity for the singular {Y}amabe problem.
\newblock {\it Indiana Univ. Math. J.}, 40(4):1277--1299, 1991.



\bibitem{WillConj}
F.C. Marques, A. Neves. Min-max theory and the Willmore conjecture. {\it Ann. of Math. (2)} 179: 583--782, (2014).





\bibitem{Polyakov}
A.~Polyakov.
\newblock Fine structure of strings.
\newblock {\it Nuclear Phys.} B 268(2):406--412, 1986.






\bibitem{Thomas}
T.Y. Thomas, On conformal geometry. {\it Proc. Natl. Acad. Sci. USA} 12: 352--359, 1926.


\bibitem{Vyatkin}
Y.~Vyatkin.
\newblock {\it Manufacturing conformal invariants of hypersurfaces}.
\newblock PhD thesis, University of Auckland, 2013.


\bibitem{Willmore}
T.J.Willmore. {\it An. Şti. Univ. “Al. I. Cuza” Iaşi Secţ. I a Mat.} 11B: 493--496, (1965).

\bibitem{Yamabe}
H. Yamabe. On a deformation of Riemannian structures on compact manifolds. {\it Osaka Math. J.} 12(1): 21-37, (1960).

\end{thebibliography}
\end{document}